\pdfoutput=1

\documentclass[english]{jnsao_global}
\usepackage[utf8]{inputenc}
\usepackage{enumitem}

\usepackage[nameinlink,capitalize]{cleveref}

\renewcommand{\AA}{\mathcal{A}}
\newcommand{\anni}{^\annisign}
\newcommand{\annisign}{\perp}
\DeclareMathOperator*{\argmin}{arg\,min}
\newcommand{\BB}{\mathcal{B}}
\newcommand{\CC}{\mathcal{C}}
\renewcommand{\d}{\mathrm{d}}
\renewcommand{\div}{\operatorname{div}}
\newcommand{\dualspace}{^\dualsign}
\newcommand{\dualsign}{\star}
\newcommand{\graph}{\operatorname{graph}}
\newcommand{\id}{\operatorname{id}}
\def\itemref#1{\itemrefintern#1MYENDING}
\def\itemrefintern#1:#2:#3MYENDING{\cref{#1:#2}\ref{#1:#2:#3}}
\newcommand{\KK}{\mathcal{K}}
\newcommand\mto{\rightrightarrows}
\newcommand{\N}{\mathbb{N}}
\newcommand{\NN}{\mathcal{N}}
\DeclareMathAlphabet{\mathpzc}{OT1}{pzc}{m}{it}
\newcommand{\oo}{\mathpzc{o}}
\newcommand{\polar}{^\polarsign}
\newcommand{\polarsign}{\circ}
\newcommand{\Proj}{\operatorname{Proj}}
\newcommand{\prox}{\operatorname{prox}}
\newcommand{\R}{\mathbb{R}}
\newcommand{\UU}{\mathcal{U}}
\newcommand{\weakly}{\rightharpoonup}
\newcommand{\YY}{\mathcal{Y}}
\newcommand{\ZZ}{\mathcal{Z}}

\newcommand{\graphlim}{\operatorname{graph~lim}}
\newcommand{\shrink}{\operatorname{shrink}}
\newcommand{\sign}{\operatorname{sign}}

\newenvironment{msc}%
{\vspace*{-0.75cm}\small\quotation\noindent{\normalfont\sectfont\nobreak MSC (2020)\quad}}{\endquotation\vskip0.7cm}

\newcommand{\mscLink}[1]{\href{https://www.ams.org/mathscinet/msc/msc2020.html?t=#1}{#1}}

\newenvironment{keywords}%
{\vspace*{-0.75cm}\small\quotation\noindent{\normalfont\sectfont\nobreak Keywords\quad}}{\endquotation\vskip0.7cm}

\DeclarePairedDelimiter\abs{\lvert}{\rvert}
\DeclarePairedDelimiter\norm{\lVert}{\rVert}

\DeclarePairedDelimiter\parens{(}{)}

\providecommand\given{\nonscript\;\delimsize|\nonscript\;\mathopen{}}
\DeclarePairedDelimiterX\set[1]{\{}{\}}{#1}

\DeclarePairedDelimiterX\seq[1]{(}{)}{#1}

\DeclarePairedDelimiterX\innerprod[2]{(}{)}{#1,#2}
\DeclarePairedDelimiterX\dual[2]{\langle}{\rangle}{#1,#2}

\theoremstyle{definition}
\newtheorem{assumption}[theorem]{Assumption}
\crefname{assumption}{Assumption}{Assumptions} 

\title{From resolvents to generalized equations and quasi-variational inequalities: existence and differentiability}
\shorttitle{From resolvents to GEs and QVIs}

\author{Gerd Wachsmuth%
	\thanks{%
		Brandenburgische Technische Universität Cottbus--Senftenberg,
		Institute of Mathematics,
		03046 Cottbus,
		Germany,
		\email{gerd.wachsmuth@b-tu.de},
		\url{https://www.b-tu.de/fg-optimale-steuerung},
		\orcid{0000-0002-3098-1503}%
	}
}
\shortauthor{Wachsmuth}

\acknowledgements{
	The author would like to thank the two referees for their careful reading of the manuscript
	and the very helpful suggestions.
	This work has been supported by the DFG Grant \emph{Bilevel Optimal Control: Theory, Algorithms, and Applications}
	(Grant No.\ WA 3636/4-2) within the Priority Program SPP 1962
	(Non-smooth and Complementarity-based Distributed Parameter Systems: Simulation and Hierarchical Optimization).
}

\manuscriptsubmitted{2021-09-30}
\manuscriptaccepted{2022-01-07}
\manuscriptvolume{3}
\manuscriptnumber{8537}
\manuscriptyear{2022}
\manuscriptdoi{10.46298/jnsao-2022-8537}

\begin{document}

\maketitle

\begin{abstract}
	We consider a generalized equation
	governed by a strongly monotone and Lipschitz
	single-valued mapping
	and a maximally monotone set-valued mapping in a Hilbert space.
	We are interested in the sensitivity of solutions
	w.r.t.\ perturbations of both mappings.
	We demonstrate that the directional differentiability
	of the solution map
	can be verified
	by using the directional differentiability
	of the single-valued operator and of the resolvent of the set-valued mapping.
	The result is applied to quasi-generalized equations
	in which we have an additional dependence
	of the solution within the set-valued part of the equation.
\end{abstract}

\begin{keywords}
	generalized equations,
	variational inclusion,
	directional differentiability,
	resolvent operator,
	proto-derivative,
	quasi-variational inequality
\end{keywords}

\begin{msc}
	\mscLink{49J53},
	\mscLink{47J22},
	\mscLink{49J50}
\end{msc}

\section{Introduction}
We consider the (local) solution mapping $S \colon U \to Y$, $u \mapsto y$, of the generalized equation
\begin{equation}
	\label{eq:p}
	0 \in A(y, u) + B(y, u).
\end{equation}
Here, $Y$ is a (real) Hilbert space,
$U$ is a (real) Banach space,
$A \colon Y \times U \to Y\dualspace$ is (locally) strongly monotone and Lipschitz w.r.t.\ its first argument
and
the set-valued map $B \colon Y \times U \mto Y\dualspace$ is assumed to be maximally monotone w.r.t.\ its first argument.

The equation \eqref{eq:p} can be used to model many real-world phenomena.
In the case where $B(\cdot, u)$ coincides with the subdifferential $\partial j_u$
of a proper, lower-semicontinuous and convex function $j_u \colon Y \to (-\infty,\infty]$,
\eqref{eq:p} is a variational inequality (VI) of the second kind.
If further, $j_u$ is the indicator function $\delta_{C_u} \colon Y \to \set{0,\infty}$
of a non-empty, closed and convex set $C_u \subset Y$,
it is a VI of the first kind
and $\partial\delta_{C_u} = \NN_{C_u}$ is the normal cone of $C_u$.

We show by some simple arguments
that
the existence of solutions and the directional differentiability of $S$
follow
from the properties of $A$ and of the resolvent $J_B$ of $B$
and from the directional differentiability of $A$ and $J_B$, respectively.
Here, $J_B \colon Y \times U \to Y$, $(q, u) \mapsto y$,
is the solution map of
\begin{equation}
	\label{eq:resol}
	0 \in R(y - q) + B(y, u)
\end{equation}
where $R \colon Y \to Y\dualspace$ is the Riesz map of the Hilbert space $Y$.
Similarly, we treat the
equation
\begin{equation}
	\label{eq:qp}
	0 \in A(y,u) + B(y - \Phi(y,u), u),
\end{equation}
where $\Phi \colon Y \times U \to Y$
is assumed to have a small Lipschitz constant w.r.t.\ the first variable.
Note that \eqref{eq:qp}
is an inclusion which generalizes
the setting of so-called quasi-variational inequalities (QVIs).

Let us put our work in perspective.
The existence of solutions to \eqref{eq:p}
is well understood,
we refer to, e.g.,
\cite[Section~23.4]{BauschkeCombettes2011}.
The first contribution
which studies differentiability of problems similar to \eqref{eq:p}
is \cite{Zarantonello1971},
in which it is shown that the projection onto a closed, convex set $C$
is directionally differentiable on $C$ itself.
In \cite{Mignot1976,Haraux1977}
it is proved that the projection is directionally differentiable everywhere,
if $C$ is assumed to be polyhedric,
see also \cite{Wachsmuth2016:2}.
Note that the projection onto $C$ is the solution mapping
of \eqref{eq:p} with
$A(y,u) = R(y - u)$
and $B(y,u) = \partial\delta_C(y) = \NN_C(y)$.
The case of non-linear $A$ was treated in \cite{LevyRockafellar1994,Levy1999}.
Later,
theory for the differentiability of $J_B$ with $B = \partial j$
was set up in \cite{Do1992}.
It was shown that the differentiability of $J_{\partial j}$
is equivalent to the
so-called twice epi-differentiability of $j$,
see also \cite{ChristofWachsmuth2017:3}.
Finally,
\cite[Theorem~1]{AdlyRockafellar2020}
studies \eqref{eq:p}
with a real parameter $u \ge 0$.
However, since we are mainly interested in directional differentiability,
this is not a restriction.

Contributions corresponding to the QVI case
\eqref{eq:qp}
are rather new
and are currently restricted to the special case
\begin{equation*}
	\text{Find } y \in K(y)
	\quad\text{s.t.}\quad
	\dual{A(y) - u}{v - y} \ge 0
	\quad\forall v \in K(y)
\end{equation*}
with $K(y) = K - \Phi(y)$ for some polyhedric set $K \subset Y$.
Important parameters for the study of this problem are
the constant $\mu_A$ of strong monotonicity of $A$
and the Lipschitz constants $L_A$ and $L_\Phi$
of $A$ and $\Phi$, respectively.
The first contribution in this direction is
\cite[Theorem~1]{AlphonseHintermuellerRautenberg2019}.
Therein,
the authors showed directional differentiability
into non-negative directions
under monotonicity assumptions on $A$
and
$L_\Phi < \mu_A / (\mu_A + L_A)$.
Afterwards,
\cite[Theorem~5.5]{Wachsmuth2019:2}
proved the directional differentiability into arbitrary directions
and the smallness assumption on $\Phi$ was relaxed to
$L_\Phi < \mu_A / L_A$
(and an even weaker inequality suffices if $A$ is the derivative of a convex function).
However, this result needs
that $\Phi$ is Fréchet differentiable
(or, at least Bouligand differentiable).
Later,
\cite[Theorem~3.2]{AlphonseHintermuellerRautenberg:2020_preprint}
showed
that it is sufficient to have a directionally differentiable $\Phi$
at the price of having again the stricter requirement
$L_\Phi < \mu_A / (\mu_A + L_A)$.
A different approach, which is based on concavity properties
and which is not restricted to the special case above,
is given in \cite{ChristofWachsmuth2021:1}.

The main contributions of our paper are the following.
\begin{enumerate}[label=(\roman*)]
	\item
		\cref{thm:diff_S} shows that the solution map of \eqref{eq:p}
		is directionally differentiable.
		This result is very similar to
		\cite[Theorem~1]{AdlyRockafellar2020}.
		However, our assumptions are localized around a solution
		and we require the directional differentiability of the resolvent $J_B$
		instead of the proto-differentiability of $B$.
		This results in a much easier proof.
	\item
		\cref{lem:proto_vs_resolvent}
		shows that the directional differentiability of $J_B$
		is equivalent to $B$ being proto-differentiable
		with a maximally monotone proto-derivative.
		Thus, the differentiability assumptions on $B$
		in \cref{thm:diff_S}
		and
		\cite[Theorem~1]{AdlyRockafellar2020}
		coincide.
	\item
		In \cref{sec:qge}
		we show that the approaches
		of
		\cite[Theorem~1]{AlphonseHintermuellerRautenberg2019}
		and
		\cite[Theorem~5.5]{Wachsmuth2019:2}
		can be generalized to deal with the solution map of \eqref{eq:qp}.
		Moreover,
		we only need directional differentiability of the data functions
		$A$, $\Phi$ and $J_B$
		as in 
		\cite{AlphonseHintermuellerRautenberg2019}
		and only the weaker requirements on $L_\Phi$
		from
		\cite{Wachsmuth2019:2}.
\end{enumerate}
The paper is structured as follows.
In \cref{sec:notation} we fix some notation
and state \cref{thm:cvx_lipschitz_2}
concerning convex functions
with a strongly monotone and Lipschitz continuous derivative.
\cref{sec:gen_eq} contains the differentiability result for \eqref{eq:p},
whereas
\cref{sec:qge} is concerned with \eqref{eq:qp}.
Some applications are presented in \cref{sec:applications}.

\section{Notation}
\label{sec:notation}

In this paper, all linear spaces are over the field of real numbers.
For a Banach space $U$,
$\norm{\,\cdot\,} \colon U \to [0,\infty)$
and
$\dual{\,\cdot\,}{\,\cdot\,} \colon U\dualspace \times U \to \R$
denote the norm and the duality product, respectively.
In a Hilbert space $Y$,
$\innerprod{\,\cdot\,}{\,\cdot\,} \colon Y \times Y \to \R$
is the inner product.
The spaces will be clear from the context
and will not be indicated by a subscript.

Let $U$ and $Y$ be a Banach space
and a Hilbert space, respectively.
We denote by $R \colon Y \to Y\dualspace$
the Riesz map of $Y$.
For $\varepsilon > 0$ and $y^* \in Y$,
$B_\varepsilon(y^*)$ ($U_\varepsilon(y^*)$) denotes the closed (open) ball
in $Y$
with center $y^*$ and radius $\varepsilon$, respectively.

If $B \colon Y \times U \mto Y\dualspace$
and
$\UU \subset U$
are given
such that
$B(\cdot, u) \colon Y \mto Y\dualspace$
is maximally monotone
for all $u \in \UU$,
we say that
$B$ is a parametrized maximally monotone operator
and
we define its resolvent
$J_B \colon Y \times \UU \to Y$
via
$J_B(\cdot, u) := J_{B(\cdot, u)}$,
i.e.,
for $(q, u) \in Y \times \UU$
the point
$y = J_B(q, u) = (R + B(\cdot, u))^{-1}(R q)$
is the unique solution of
\begin{equation*}
	0 \in R(y - q) + B(y, u)
	.
\end{equation*}

For a closed, convex subset $K \subset Y$,
we denote by $T_K \colon Y \mto Y$ and $N_K \colon Y \mto Y\dualspace$
the tangent-cone and normal-cone map.
Moreover,
by
\begin{equation*}
	K\polar := \set{ \mu \in Y\dualspace \given \dual{\mu}{v} \le 0 \; \forall v \in K },
	\qquad
	\mu\anni := \set{ v \in Y \given \dual{\mu}{v} = 0 }
\end{equation*}
we
denote the polar cone of $K$
and
the annihilator of $\mu \in Y\dualspace$,
respectively.

We need a characterization of convex functions
defined on subsets of $Y$
with a Lipschitz continuous derivative.
\begin{theorem}
	\label{thm:cvx_lipschitz}
	Let $\YY \subset Y$ be nonempty, open and convex
	and $L_f \in (0,\infty)$.
	Further,
	the function $f \colon \YY \to \R$
	is assumed to be convex and continuous.
	Then, the following assertions
	are equivalent.
	\begin{enumerate}[label=(\roman*)]
		\item
			$f$ is Gâteaux differentiable on $\YY$ and $f' \colon \YY \to Y\dualspace$ is $L_f$-Lipschitz continuous on $\YY$.
		\item
			$f$ is Gâteaux differentiable on $\YY$ and $f' \colon \YY \to Y\dualspace$ is $1/L_f$-cocoercive,
			i.e.,
			\begin{equation*}
				\dual{f'(y_2) - f'(y_1)}{y_2 - y_1} \ge \frac1{L_f}\norm{f'(y_2) - f'(y_1)}^2
				\qquad
				\forall y_1,y_2 \in \YY.
			\end{equation*}
		\item
			The map $\frac{L_f}{2} \norm{\cdot}^2 - f$
			is convex on $\YY$.
	\end{enumerate}
\end{theorem}
For the proof, we refer to \cite[Theorem~3.1]{PerezArosVilches2019}.

Next, we give
an important
inequality
for convex functions.
This inequality is well-known
in the finite-dimensional case if the convex function is defined on the entire space,
see, e.g.,
\cite[Theorem~2.1.12]{Nesterov2004}
or
\cite[Lemma~3.10]{Bubeck2015}.
The infinite dimensional version (on the entire space) was given in
\cite[Lemma~3.4]{Wachsmuth2019:2}.
By using \cref{thm:cvx_lipschitz},
we can adopt the proof to the situation at hand.
\begin{theorem}
	\label{thm:cvx_lipschitz_2}
	Let $\YY \subset Y$ be nonempty, open and convex.
	Further,
	let $f \colon \YY \to \R$
	be convex, continuous and Gâteaux differentiable
	such that
	$f' \colon \YY \to Y\dualspace$ is strongly monotone with constant $\mu_f \in (0,\infty)$
	and Lipschitz continuous with constant $L_f \in (0,\infty)$.
	Then,
	\begin{equation*}
		\dual[\big]{f'(y_2) - f'(y_1)}{y_2 - y_1}
		\ge
		\frac{\mu_f L_f}{\mu_f + L_f} \norm{y_2 - y_1}^2
		+
		\frac{1}{\mu_f + L_f} \norm[\big]{f'(y_2) - f'(y_1)}^2
	\end{equation*}
	for all $y_1, y_2 \in \YY$.
\end{theorem}
\begin{proof}
	We define $g := f - \frac{\mu_f}{2} \norm{\cdot}^2$.
	Thus, we have $g'(y) = f'(y) - \mu_f R y$.
	The strong monotonicity of $f'$ implies that $g$ is convex.
	From \cref{thm:cvx_lipschitz}
	we infer that
	$\frac{L_f}{2} \norm{\cdot}^2 - f = \frac{L_f-\mu_f}{2} \norm{\cdot}^2 - g$
	is convex.
	Applying \cref{thm:cvx_lipschitz} again
	shows that $g'$ is $1/(L_f-\mu_f)$-cocoercive,
	i.e.,
	for arbitrary $y_1, y_2 \in \YY$
	we have
	\begin{align*}
		&
		(L_f - \mu_f) \parens*{
			\dual{f'(y_2) - f'(y_1)}{y_2 - y_1}
			-
			\mu_f \norm{y_2 - y_1}^2
		}
		\\&\qquad
		=
		(L_f - \mu_f) \dual{g'(y_2) - g'(y_1)}{y_2 - y_1}
		\\&\qquad
		\ge
		\norm{g'(y_2) - g'(y_1)}^2
		=
		\norm{f'(y_2) - f'(y_1)}^2
		-
		2 \mu_f
		\dual{f'(y_2) - f'(y_1)}{y_2 - y_1}
		+
		\mu_f^2
		\norm{y_2 - y_1}^2
		.
	\end{align*}
	Rearranging terms yields
	the claim.
\end{proof}

\section{Generalized equations}
\label{sec:gen_eq}
We consider the solution mapping of the generalized equation \eqref{eq:p}.
We show that solutions
are locally stable and directionally differentiable under suitable assumptions.

\subsection{Local solvability}
\label{subsec:local_solv}
We set up the standing assumptions
which allow to prove
that \eqref{eq:p}
is uniquely solvable around
a given reference solution $(y^*, u^*)$.
\begin{assumption}[Standing assumptions]
	\label{asm:standing}
	Let $(y^*, u^*) \in Y \times U$
	be
	given
	and let $\UU \subset U$ be a neighborhood of $u^*$.
	\begin{enumerate}[label=(\roman*)]
		\item
			\label{asm:standing:1}
			For all $u \in \UU$,
			$A(\cdot,u)$ is locally strongly monotone and Lipschitz
			in a neighborhood of $y^*$, uniformly in $u \in \UU$.
			To be precise, there exist constants $\varepsilon, \mu_A, L_A > 0$
			such that
			for all $y_1, y_2 \in B_{\varepsilon}(y^*)$
			and for all $u \in \UU$
			we have
			\begin{align*}
				\dual{ A(y_2, u) - A(y_1, u) }{ y_2 - y_1} &\ge \mu_A \norm{y_2 - y_1}^2,
				\\
				\norm{ A(y_2, u) - A(y_1, u)}
				&
				\le L_A \norm{y_2 - y_1}.
			\end{align*}
		\item
			\label{asm:standing:2}
			$B(\cdot, u) \colon Y \mto Y\dualspace$ is maximally monotone
			for all $u \in \UU$.
	\end{enumerate}
\end{assumption}
Here and in the sequel,
it is sufficient that $A$ is defined only on $B_\varepsilon(y^*) \times \UU$
and that $B$ is defined on $Y \times \UU$.

We check that the GE
has a unique solution in a neighborhood of $y^*$
for certain ``small'' perturbations $u$ of $u^*$.
\begin{theorem}
	\label{thm:solvability_vi}
	Let \cref{asm:standing} be satisfied
	by a solution $(y^*, u^*)$ of \eqref{eq:p}.
	We select $\rho \in (0, 2 \mu_A / L_A^2)$,
	$r \in (0,\varepsilon]$
	and set $c := \sqrt{1 - 2 \rho \mu_A + \rho^2 L_A^2} \in (0,1)$.
	Suppose that $\zeta \in Y\dualspace$
	and $u \in \UU$
	satisfy
	\begin{equation}
		\label{eq:est}
		\norm{
			J_{\rho B}(q_\rho^*, u)
			-
			J_{\rho B}(q_\rho^*, u^*)
		}
		+
		\rho \norm{A(y^*, u) - A(y^*, u^*)}
		+
		\rho \norm{\zeta}
		\le
		(1 - c) r
		,
	\end{equation}
	where
	$q_\rho^* = y^* - \rho R^{-1} A(y^*, u^*)$.
	Then,
	there exist unique solutions $y, \tilde y$
	of
	\begin{equation}
		\label{eq:sol_zeta_u}
		0 \in \zeta + A(y,u) + B(y,u),
		\quad y \in B_\varepsilon(y^*)
	\end{equation}
	and
	\begin{equation}
		\label{eq:sol_u}
		0 \in A(\tilde y,u) + B(\tilde y,u),
		\quad \tilde y \in B_\varepsilon(y^*)
		.
	\end{equation}
	These solutions satisfy
	$y, \tilde y \in B_r(y^*)$
	and
	$\norm{y - \tilde y} \le \frac{\rho}{1 - c} \norm{\zeta}$.
\end{theorem}
Assumption \eqref{eq:est} ensures that \eqref{eq:sol_zeta_u} and \eqref{eq:sol_u}
are small perturbations of \eqref{eq:p} with $u = u^*$.
Later, we will see that continuity assumptions on $A$ and on (the resolvent of) $B$
can be used to verify \eqref{eq:est}, see \cref{lem:local_solve}.
\begin{proof}
	The proof is inspired by the proof of \cite[Theorem~5.6]{Brezis2011}
	which treats the special case of a variational inequality.

	We define $T \colon Y \to Y$, $y \mapsto z$
	as the solution operator of
	\begin{equation*}
		0 \in R ( z - (y - \rho R^{-1} A(y, u) - \rho R^{-1} \zeta )) + \rho B(z,u)
		,
	\end{equation*}
	i.e.,
	\begin{equation*}
		T(y)
		=
		J_{\rho B}\parens{
			y - \rho R^{-1} A(y, u) - \rho R^{-1} \zeta, u
		}
		.
	\end{equation*}
	In order to apply the Banach fixed-point theorem,
	we check that
	the mapping
	$Q \colon B_\varepsilon(y^*) \to Y$, $x \mapsto x - \rho R^{-1} A(x, u)$
	is a contraction.
	Indeed,
	\begin{align*}
		\norm{
			x - \rho R^{-1} A(x, u)
			- y + \rho R^{-1} A(y, u)
		}^2
		&=
		\norm{x - y}^2 + \norm{\rho R^{-1} (A(x,u) - A(y,u))}^2
		\\&\qquad
		-
		2 \rho \dual{A(x,u) - A(y,u)}{x - y}
		\\&
		\le
		\parens*{
			1 - 2 \rho \mu_A + \rho^2 L_A^2
		}
		\norm{x - y}^2
	\end{align*}
	holds for all $x,y \in B_\varepsilon(y^*)$
	due to \itemref{asm:standing:1}.
	Hence, $Q$ is Lipschitz with constant $c \in (0,1)$.
	Since $J_{\rho B}(\cdot, u)$
	is Lipschitz with constant $1$,
	$T \colon B_\varepsilon(y^*) \to Y$ is a contraction.

	It remains to check that $T$ maps $B_r(y^*)$
	onto itself.
	To this end, we denote the three terms
	on the left-hand side of \eqref{eq:est}
	by $\kappa_1$, $\kappa_2$ and $\kappa_3$, respectively.
	By using $y^* = J_{\rho B}(q_\rho^*, u^*)$,
	we have for an arbitrary $y \in B_r(y^*)$
	\begin{align*}
		\norm{T(y) - y^*}
		&=
		\norm{
			J_{\rho B}\parens{ Q(y) - \rho R^{-1} \zeta, u }
			-
			J_{\rho B}(q_\rho^*, u^*)
		}
		\\
		&\le
		\norm{
			J_{\rho B}\parens{ Q(y) - \rho R^{-1} \zeta, u }
			-
			J_{\rho B}(q_\rho^*, u)
		}
		+
		\norm{
			J_{\rho B}(q_\rho^*, u)
			-
			J_{\rho B}(q_\rho^*, u^*)
		}
		\\
		&\le
		\norm{
			y - \rho R^{-1} A(y, u) - \rho R^{-1} \zeta
			-
			(y^* - \rho R^{-1} A(y^*, u^*) )
		}
		+
		\kappa_1
		\\
		&\le
		\norm{
			y - \rho R^{-1} A(y, u)
			-
			(y^* - \rho R^{-1} A(y^*, u) )
		}
		+
		\rho
		\norm{
			A(y^*, u)
			-
			A(y^*, u^*)
		}
		+
		\rho \norm{\zeta}
		+
		\kappa_1
		\\
		&=
		\norm{Q(y) - Q(y^*)}
		+
		\kappa_1 + \kappa_2 + \kappa_3
		\le
		r
		.
	\end{align*}
	Hence, we have shown that $T \colon B_r(y^*) \to B_r(y^*)$
	is a contraction.

	The Banach fixed-point theorem yields a unique fixed point $y \in B_r(y^*)$
	and this is a solution of \eqref{eq:sol_zeta_u}.
	Similarly, the solvability of \eqref{eq:sol_u}
	is obtained by using the same arguments with $\zeta = 0$.
	By repeating the same argument with $r = \varepsilon$,
	we establish the uniqueness of solutions in $B_\varepsilon(y^*)$.

	For the final estimate,
	we note
	\begin{align*}
		\tilde y &= J_{\rho B}( \tilde y - \rho R^{-1} A(\tilde y,u), u),
		&
		y &= J_{\rho B}( y - \rho R^{-1} A(y,u) - \rho R^{-1} \zeta, u ).
	\end{align*}
	Thus,
	\begin{equation*}
		\norm{y - \tilde y}
		\le
		\norm{Q(y) - Q(\tilde y)} + \rho \norm{\zeta}
		\le
		c \norm{y - \tilde y} + \rho \norm{\zeta}
	\end{equation*}
	and this gives
	the desired estimate.
\end{proof}
The second term on the left-hand side of \eqref{eq:est}
becomes small if we assume some continuity of $A$
and if $u$ is close to $u^*$.
In order to control the first term,
we use a famous result by Attouch.
\begin{lemma}[{\cite[Proposition~3.60]{Attouch1984}}]
	\label{lem:conv_resolvents}
	Let $\seq{B_n}_{n \in \N}$ be a sequence of maximal monotone operators on $Y$
	and let $B_0 \colon Y \mto Y\dualspace$ be maximally monotone.
	Then, the following are equivalent.
	\begin{enumerate}[label=(\roman*)]
		\item
			There exists $\rho_0 > 0$ such that for all $y \in Y$ we have $J_{\rho_0 B_n}(y) \to J_{\rho_0 B_0}(y)$.
		\item
			For all $\rho > 0$ and all $y \in Y$ we have $J_{\rho B_n}(y) \to J_{\rho B_0}(y)$.
	\end{enumerate}
\end{lemma}
\begin{corollary}
	\label{cor:cts_B}
	Suppose that for every $v \in Y$, the mapping $J_{B}(v, \cdot) \colon \UU \to Y$
	is continuous at $u^*$.
	Then, for all $\rho > 0$ and $v \in Y$,
	$J_{\rho B}(v, \cdot) \colon \UU \to Y$
	is continuous at $u^*$.
\end{corollary}
\begin{proof}
	We have to show that $J_{\rho B}(v, u_n) \to J_{\rho B}(v, u^*)$
	for all sequences
	$\seq{u_n}_{n \in \N} \subset \UU$ with $u_n \to u^*$.
	This follows directly from \cref{lem:conv_resolvents}
	by setting $B_n := B(\cdot, u_n)$ and $B_0 := B(\cdot, u^*)$.
\end{proof}
Hence, if
\cref{asm:standing}
and the assumption of \cref{cor:cts_B}
are satisfied
and if $A(y^*, \cdot) \colon \UU \to Y$ is continuous at $u^*$,
the estimate \eqref{eq:est} holds if $\zeta$ is small enough
and if $u$ is sufficiently close to $u^*$.
For later reference, we also give a directional version of this statement.
\begin{lemma}
	\label{lem:local_solve}
	Let \cref{asm:standing} be satisfied
	by a solution $(y^*, u^*)$ of \eqref{eq:p}
	and choose $\rho, r, c$ as in \cref{thm:solvability_vi}.
	Further, let $h \in U$ be arbitrary.
	We assume that
	$t \mapsto A(y^*, u^* + t h) \in Y\dualspace$
	and
	$t \mapsto J_B(v, u^* + t h) \in Y$
	are right-continuous at $t = 0$ for all $v \in Y$.
	Then, the estimate \eqref{eq:est}
	is satisfied by $u = u^* + t h$
	and $\zeta \in Y\dualspace$
	if $t > 0$ and $\norm{\zeta}$
	are small enough.
\end{lemma}

\subsection{Directional differentiability}
\label{subsec:dir_dif}
In order to prove directional differentiability of the solution mapping,
we need some differentiability assumptions concerning the
mappings $A$ and $B$.

\begin{assumption}[Differentiability assumptions]
	\label{asm:diff}
	In addition to \cref{asm:standing},
	we suppose the following.
	\begin{enumerate}[label=(\roman*)]
		\item
			\label{asm:diff:1}
			$A$ is directionally differentiable at $(y^*, u^*)$.
		\item
			\label{asm:diff:2}
			$J_B$ is directionally differentiable at $(q^*, u^*)$
			with $q^* = y^* - R^{-1} A(y^*, u^*)$.
	\end{enumerate}
\end{assumption}

Interestingly,
the next result shows that
\itemref{asm:diff:2}
already implies
that the directional derivative of $J_B$
is again a resolvent of a parametrized maximally monotone operator.
In the setting that $B$ is a subdifferential and independent of $u$,
this result follows from \cite[Theorems~3.9, 4.3]{Do1992}.
\begin{lemma}
	\label{lem:max_mon_dif_dep}
	Let \cref{asm:diff} be satisfied
	by a solution $(y^*, u^*)$ of \eqref{eq:p}.
	We further set
	\(
		\xi^*
		:=
		-A(y^*, u^*) \in B(y^*, u^*)
	\).
	Then,
	the operator
	$D B (y^*, u^* \mathbin{|} \xi^* ) \colon Y \times U \mto Y\dualspace$,
	\begin{equation}
		\label{eq:def_db}
		DB(y^*, u^* \mathbin{|} \xi^*)(\delta, h)
		:=
		\set{
			R ( k - \delta )
			\given
			k \in Y, \delta = J_B'(q^*, u^*; k, h)
		}
		\qquad
		\forall (\delta,h) \in Y \times U
		,
	\end{equation}
	is
	a
	parametrized maximally monotone
	operator
	and we have
	\begin{equation}
		\label{eq:deriv_is_resolv}
		J_B'(q^*, u^*; \cdot)
		=
		J_{D B (y^*, u^* \mathbin{|} \xi^* )},
	\end{equation}
	i.e., $\delta = J_B'(q^*, u^*; k, h) = J_{D B (y^*, u^* \mathbin{|} \xi^* )}(k, h)$
	if and only if $\delta$ solves
	\begin{equation*}
		0 \in R( \delta - k ) + D B (y^*, u^* \mathbin{|} \xi^* )(\delta, h).
	\end{equation*}
\end{lemma}
We have chosen such a complicated name
for the linearization of $B$,
since it will turn out that
this linearization coincides with the so-called
proto-derivative of $B$,
see \cref{subsec:proto}.
\begin{proof}
	Let us check the monotonicity w.r.t.\ the parameter $\delta$.
	For fixed $h \in U$ and arbitrary $k_1, k_2 \in Y$, we have
	\begin{align*}
		&
		\norm*{
			\frac{J_B(q^* + t k_1, u^* + t h) - J_B(x)}{t}
			-
			\frac{J_B(q^* + t k_2, u^* + t h) - J_B(x)}{t}
		}^2
		\\
		&\qquad=\frac1{t^2}
		\norm{J_B(q^* + t k_1, u^* + t h) - J_B(q^* + t k_2, u^* + t h)}^2
		\\
		&\qquad\le
		\frac1{t^2}
		\innerprod{J_B(q^* + t k_1, u^* + t h) - J_B(q^* + t k_2, u^* + t h)}{t \, (k_1 - k_2)}
		\\
		&\qquad=
		\innerprod*{
			\frac{J_B(q^* + t k_1, u^* + t h) - J_B(x)}{t}
			-
			\frac{J_B(q^* + t k_2, u^* + t h) - J_B(x)}{t}
		}{k_1 - k_2}
		.
	\end{align*}
	Here,
	the inequality follows
	from the fact that resolvents
	of maximally monotone operators
	are firmly non-expansive.
	By taking the limit
	$t \searrow 0$, we find
	\begin{equation*}
		\norm{ J_B'(q^*, u^*; k_1, h) - J_B'(q^*, u^*; k_2, h)}^2
		\le
		\innerprod{
			J_B'(q^*, u^*; k_1, h) - J_B'(q^*, u^*; k_2, h)
		}{ k_1 - k_2 }
		.
	\end{equation*}
	Hence,
	\begin{equation*}
		\innerprod*{
			J_B'(q^*, u^*; k_1, h) - J_B'(q^*, u^*; k_2, h)
		}
		{
			k_1 - J_B'(q^*, u^*; k_1 , h)
			-
				k_2 + J_B'(q^*, u^*; k_2 , h)
		}
		\ge
		0,
	\end{equation*}
	i.e.,
	\begin{equation*}
		\dual*{
			R ( k_1 - \delta_1 ) - R ( k_2 - \delta_2 )
		}
		{
			\delta_1 - \delta_2
		}
		\ge
		0,
	\end{equation*}
	where $\delta_i := J_B'(q^*, u^*; k_i, h)$.
	This shows monotonicity of $DB(y^*, u^* \mathbin{|} \xi^*)(\cdot, h)$.

	Minty's theorem, see \cite[Theorem~21.1]{BauschkeCombettes2011},
	implies that the operator $DB(y^*, u^* \mathbin{|} \xi^*)(\cdot, h)$
	is maximally monotone
	if and only if
	$R + DB(y^*, u^* \mathbin{|} \xi^* )(\cdot, h)$
	is surjective.
	This, however,
	is obvious:
	for an arbitrary $\zeta \in Y\dualspace$,
	we can set $\delta = J_B'(q^*, u^*; R^{-1}\zeta, h)$
	and have
	$\zeta \in R \delta + DB(y^*, u^* \mathbin{|} \xi^* )(\delta, h)$.

	Finally, \eqref{eq:deriv_is_resolv} follows from a straightforward calculation.
\end{proof}
\begin{remark}
	\label{rem:weak_derivative}
	Minimal changes to the proof
	show
	that
	the assertion of \cref{lem:max_mon_dif_dep}
	remains true if we only assume that $J_B$ is weakly directionally differentiable
	at $(q^*, u^*)$,
	i.e., if
	\begin{equation*}
		\frac{ J_B( q^* + t k, u^* + t h) - J( q^*, u^* )}{t}
		\weakly
		J_B'(q^*, u^*; k, h)
		\qquad\text{in $Y$ as $t \searrow 0$}
	\end{equation*}
	for all $(k,h) \in Y \times U$.
\end{remark}

Next, we apply \cref{lem:max_mon_dif_dep}
to the normal cone mapping of a polyhedric set.
\begin{proposition}
	\label{prop:polyhedric}
	Suppose that $K \subset Y$ is polyhedric
	and $B(\cdot, u) := N_K$ is the normal cone mapping to $K$
	(independent of $u$).
	Then, \itemref{asm:diff:2} is satisfied
	and we have
	\begin{equation*}
		DB( y^* \mathbin{|} \xi^* )(\delta)
		=
		N_{\KK}(\delta)
		=
		\begin{cases}
			\KK\polar \cap \delta\anni & \text{if } \delta \in \KK, \\
			\emptyset & \text{if } \delta \not\in \KK,
		\end{cases}
	\end{equation*}
	where $\KK = T_K(y^*) \cap (\xi^*)\anni$
	denotes the critical cone.
	That is, $DB(y^* \mathbin{|} \xi^*)$
	is the normal cone mapping to the critical cone $\KK$.
	Here, we have suppressed the arguments $u^*$ and $h$.
\end{proposition}
\begin{proof}
	From \cite{Mignot1976,Haraux1977}
	we get the directional differentiability
	of $J_B = \Proj_K$
	and
	\begin{equation*}
		\Proj_K'(q^*; k)
		=
		\Proj_{\KK}(k)
		.
	\end{equation*}
	Using \eqref{eq:def_db},
	we have
	\begin{align*}
		DB(y^* \mathbin{|} \xi^*)(\delta)
		&:=
		\set{
			R ( k - \delta )
			\given
			k \in Y, \delta = \Proj_{\KK}(k)
		}
		\\
		&\phantom{:}=
		\set{
			R ( k - \delta )
			\given
			k \in Y,
			\delta \in \KK,
			R( k - \delta ) \in \KK\polar,
			\dual{ R(k - \delta)}{\delta} = 0
		}
	\end{align*}
	and the claim follows.
\end{proof}

\begin{theorem}
	\label{thm:diff_S}
	Suppose that \cref{asm:diff} is satisfied
	by a solution $(y^*, u^*)$ of \eqref{eq:p}.
	We denote by $S$ the local solution mapping of \eqref{eq:p},
	cf.\ \cref{thm:solvability_vi,lem:local_solve}.
	Then, $S$ is directionally differentiable at $u^*$.
	For
	$h \in U$,
	the derivative
	$\delta = S'(u^*; h)$
	is the unique solution of
	\begin{equation}
		\label{eq:linearized}
		0 \in A'(y^*,u^*; \delta, h) +
		DB(y^*, u^* \mathbin{|} \xi^* )(\delta, h)
		,
	\end{equation}
	where $\xi^* = -A(y^*, u^*)$.
\end{theorem}
\begin{proof}
	The proof is inspired by \cite{Levy1999}.
	Since $A'(y^*, u^*; \cdot, h) \colon Y \to Y\dualspace$
	is again strongly monotone and Lipschitz,
	and since the operator $DB(y^*, u^* \mathbin{|} \xi^* )(\cdot, h) \colon Y \mto Y\dualspace$
	is maximally monotone by \cref{lem:max_mon_dif_dep},
	it is clear that the linearized equation
	possesses a unique solution $\delta \in Y$
	for an arbitrary $h \in U$.
	It remains to check that $\delta$
	is the directional derivative of $S$.
	We set
	\begin{align*}
		q^* &:= y^* - R^{-1} A(y^*,u^*),
		&
		k &:= \delta - R^{-1} A'(y^*,u^*; \delta, h).
	\end{align*}
	This implies $y^* = J_B(q^*,u^*)$
	and
	$\delta = J_{B}'(q^*,u^*; k, h)$,
	cf.\ \eqref{eq:deriv_is_resolv}.

	Next, we take an arbitrary sequence $\seq{t_n}_{n \in \N} \subset (0,\infty)$ with $t_n \searrow 0$
	and define
	\begin{equation*}
		y_n := J_B( q^* + t_n k , u^* + t_n h),
		\qquad\text{i.e.,}\quad
		0 \in R(y_n - q^* - t_n k) + B(y_n, u^* + t_n h)
	\end{equation*}
	for $n$ large enough.
	Then,
	\itemref{asm:diff:2}
	implies
	\begin{equation*}
		\frac{y_n - y^*}{t_n}
		\to
		J_B'(q^*, u^*; k, h)
		=
		\delta.
	\end{equation*}
	Now, we define
	\begin{equation*}
		\zeta_n := \frac{R(y_n - q^* - t_n k) - A(y_n, u^* + t_n h)}{t_n}
		.
	\end{equation*}
	By using the definition of $q^*$
	we get
	\begin{align*}
		\zeta_n
		&= \frac{R(y_n - y^* - t_n k) + A(y^*,u^*) - A(y_n, u^* + t_n h)}{t_n}
		\\
		&=
		R\parens[\Big]{\frac{y_n - y^*}{t_n} - k}
		+
		\frac{A(y^*,u^*) - A(y^* + t_n \delta, u^* + t_n h)}{t_n}
		+
		\frac{A(y^* + t_n \delta, u^* + t_n h) - A(y_n, u^* + t_n h)}{t_n}
		.
	\end{align*}
	Due to \itemref{asm:standing:1},
	the last addend can be bounded by $L_A \norm{(y_n - y^*)/t_n - \delta} \to 0$.
	Thus,
	the directional differentiability of $A$ implies
	\(
		\zeta_n
		\to
		R (\delta - k )
		-
		A'(y^*, u^*, \delta, h)
		=
		0
	\)%
	.
	Next,
	we note that the definition of $\zeta_n$ yields
	\begin{equation*}
		0 \in t_n \zeta_n + A(y_n, u^* + t_n h) + B(y_n, u^* + t_n h).
	\end{equation*}
	Due to \cref{lem:local_solve},
	we can apply \cref{thm:solvability_vi}
	for $n$ large enough.
	This yields the existence
	of a unique solution $\tilde y_n \in B_{\varepsilon}(y^*)$
	of
	\begin{equation*}
		0 \in A(\tilde y_n, u^* + t_n h) + B(\tilde y_n, u^* + t_n h)
	\end{equation*}
	and this solution satisfies $\norm{y_n - \tilde y_n} \le C \norm{t_n \zeta_n}$.
	Note that $\tilde y_n = S(u^* + t_n h)$.
	Hence,
	\begin{equation*}
		\frac{S(u^* + t_n h ) - S(u^*)}{t_n}
		=
		\frac{\tilde y_n - y_n}{t_n} + \frac{y_n - y^*}{t_n}
		\to
		0
		+
		\delta.
	\end{equation*}
	This shows the claim.
\end{proof}

We mention that a similar result
has been given in
\cite[Theorem~1]{AdlyRockafellar2020}.
Therein, global assumptions on $A$ are used,
i.e.,
\itemref{asm:standing:1}
is required to hold for $\varepsilon = \infty$.
Moreover,
this contribution
uses the concept of proto-differentiability,
which is not utilized in our proof.
In particular, instead of \itemref{asm:diff:2},
\cite{AdlyRockafellar2020}
requires that $B$ is proto-differentiable
with a maximally monotone proto-derivative.

In our opinion,
it is often easier and more natural
to study the differentiability properties of $J_B$
instead of checking whether $B$ is proto-differentiable.
Indeed,
in the case that $B$ is the normal cone mapping
of a polyhedric set $C \subset Y$,
the directional differentiability
of $J_B = \Proj_C$
was already shown in \cite{Mignot1976,Haraux1977},
whereas the proto-differentiability
of $B$ was proved later in
\cite[Example~4.6]{Do1992},
see also
\cite{Levy1999}.
Moreover,
the former proofs are rather elementary,
whereas the latter proof
utilizes \cite[Theorem~3.9]{Do1992}
which is based on Attouch's theorem
linking the Mosco-convergence of convex functions
with the graphical convergence of their subdifferentials.

\begin{remark}
	\label{rem:on_some_theorem}
	We comment on some extensions and limitations of \cref{thm:diff_S}.
	\begin{enumerate}[label=(\roman*)]
		\item
			The strong monotonicity of $A$ can be replaced
			by requiring
			that the linearized equation \eqref{eq:linearized}
			possesses solutions for all $h \in U$
			and
			by assuming that
			the assertion of \cref{thm:solvability_vi}
			holds.
		\item
			It is not possible to adapt the proof to the situation of \cref{rem:weak_derivative}.
			Indeed, if we only assume weak directional differentiability of $J_B$,
			we only get $(y_n - y^*)/t_n \weakly \delta$
			and, thus,
			only $\zeta_n \weakly 0$
			(if $A$ is Bouligand differentiable).
			This, however, is not enough
			to obtain $\norm{y_n - \tilde y_n} = \oo(t_n)$
			in the last step of the proof.
		\item
			Another approach for proving \cref{thm:diff_S}
			is to directly consider $\tilde y_n := S(u^* + t_n h)$.
			Due to the Lipschitz continuity of $S$,
			one gets $(\tilde y_n - y^*) / t_n \weakly \delta$ along a subsequence
			for some $\delta \in Y$.
			The next step would be to perform a Taylor expansion of
			$A(\tilde y_n, u^* + t_n h) - A(y^*, u^*)$,
			but this needs stronger differentiability assumptions for $A$,
			e.g., Bouligand/Fréchet differentiability w.r.t.\ its first argument,
			cf.\ \cite[Remark~2.3(iii), Theorem~2.14]{ChristofWachsmuth2017:3}.
			In the above proof, this is avoided via the construction of $y_n$.
	\end{enumerate}
\end{remark}

Finally,
we mention that the directional differentiability of
$J_{\rho_0 B}$ for some $\rho_0 > 0$
implies the directional differentiability
for all $\rho > 0$.
\begin{corollary}
	\label{cor:dir_dif_res}
	Let $B \colon Y \times \UU \mto Y\dualspace$ as in 
	\itemref{asm:standing:2}
	be given and fix $(q^*, u^*) \in Y \times \UU$.
	With $y^* = J_B(q^*, u^*)$,
	the following are equivalent.
	\begin{enumerate}[label=(\roman*)]
		\item
			There exists $\rho_0 > 0$ such that
			the resolvent
			$J_{\rho_0 B}$ is directionally
			differentiable at $(y^* + \rho_0 (q^* - y^*), u^*)$.
		\item
			For all $\rho > 0$,
			$J_{\rho B}$ is directionally
			differentiable at $(y^* + \rho (q^* - y^*), u^*)$.
	\end{enumerate}
\end{corollary}
\begin{proof}
	This follows from \cref{thm:diff_S},
	since $J_{\rho B} \colon (x,u) \mapsto y$
	is the solution map of
	\(
		0
		\in
		\frac{\rho_0}{\rho} R(y - x) + \rho_0 B(y, u)
	\)
	and
	$y^* = J_{\rho B}( y^* + \rho (q^* - y^*), u^*)$.
\end{proof}
Similarly,
an application of \cref{thm:diff_S} shows
that the directional differentiability of the resolvent $J_B$
is actually independent of the Riesz isomorphism $R$ of $Y$
and, thus, independent of the inner product in $Y$.

\subsection{Relation to proto-differentiability}
\label{subsec:proto}
The purpose of this section is to shed some light
on the relation of directional differentiability of the resolvent $J_B$
and
the proto-differentiability of $B$.

We first fix the notion of proto-differentiability
of the parametrized set-valued map $B$.
\begin{definition}
	\label{def:proto}
	Let $(y^*, u^*) \in Y \times U$ be given,
	such that \itemref{asm:standing:2}
	is satisfied.
	For some $\xi^* \in B(y^*, u^*)$
	and $(\delta, h) \in Y \times U$,
	we define
	\begin{equation}
		\label{eq:delta_t_B}
		\Delta_t B ( y^*, u^* \mathbin{|} \xi^*)(\delta, h)
		:=
		\frac{ B(y^* + t \delta, u^* + t h) - \xi^*}{t}
	\end{equation}
	for $t > 0$ small enough.
	We say that
	$B$ is proto-differentiable at $(y^*, u^*)$
	relative to $\xi^* \in B(y^*, u^*)$
	if the graph of
	$\Delta_t B ( y^*, u^* \mathbin{|} \xi^*)(\cdot, h) \colon Y \mto Y\dualspace$
	converges
	as $t \searrow 0$
	in the sense of Painlevé--Kuratowski,
	see \cite[Definition~2]{AdlyRockafellar2020},
	for all $h \in U$.
	In this case,
	we define its proto-derivative
	$DB(y^*, u^* \mathbin{|} \xi^*) \colon Y \times U \mto Y\dualspace$
	via
	\begin{equation*}
		\graph DB(y^*, u^* \mathbin{|} \xi^*)(\cdot, h)
		:=
		\graphlim_{t \searrow 0} \Delta_t B(y^*, u^* \mathbin{|} \xi^*)(\cdot, h)
	\end{equation*}
	for all $h \in U$.
\end{definition}
Although the operator
$\Delta_t B(y^*, u^* \mathbin{|} \xi^*)(\cdot, h)$
defined in \eqref{eq:delta_t_B}
is maximally monotone,
cf.\ \cite[Lemma~1]{AdlyRockafellar2020},
its graphical limit
$DB(y^*, u^* \mathbin{|} \xi^*)(\cdot, h)$
might
fail to be maximally monotone,
even if it exists,
see
\cite[Theorem~2]{Wachsmuth2021:1}.

\begin{lemma}
	\label{lem:proto_vs_resolvent}
	Let $(y^*, u^*) \in Y \times U$ be given,
	such that \itemref{asm:standing:2}
	is satisfied.
	Further, let $q^* \in Y$
	be given such that $y^* = J_B(q^*, u^*)$
	and we set
	$\xi^* := R(q^* - y^*) \in B(y^*, u^*)$.
	Then, the following are equivalent.
	\begin{enumerate}[label=(\roman*)]
			\item
				The mapping $J_B$
				is directionally differentiable at $(q^*, u^*)$.
			\item
				The mapping $B$ is proto-differentiable
				at $(y^*, u^*)$ relative to $\xi^*$
				and
				its proto-derivative
				$DB(y^*, u^* \mathbin{|} \xi^*)(\cdot, h)$
				is maximally monotone for all $h \in U$.
		\end{enumerate}
\end{lemma}
\begin{proof}
	By defining
	$A(y,u) := R(y - q^*)$
	it can be easily checked that $(y^*, u^*)$ is a solution of \eqref{eq:p}
	and that \cref{asm:diff} is satisfied.

	In order to apply the results from \cite{AdlyRockafellar2020},
	we fix $h \in U$ and $t_0 > 0$ such that $u^* + t h \in \UU$ for all $t \in [0,t_0)$.
	We define
	\begin{equation*}
		\tilde B_h(t, y) := B(y, u^* + t h)
		\qquad\forall t \in [0,t_0), y \in Y.
	\end{equation*}
	Now, it can be checked that
	$\Delta_t B(y^*, u^* \mathrel{|} \xi^*)(\cdot, h)$
	coincides with
	$\Delta_t \tilde B_h(y^* \mathrel{|} \xi^*)(\cdot)$
	as defined in \cite[(9)]{AdlyRockafellar2020}.
	Thus,
	$B$ is proto-differentiable at $(y^*, u^*)$
	relative to $\xi^*$
	if and only if
	$\tilde B_h$ is proto-differentiable at $y^*$
	relative to $\xi^*$
	for all $h \in U$
	and we have the formula
	\begin{equation*}
		DB(y^*, u^* \mathbin{|} \xi^*)(\cdot, h)
		=
		D \tilde B_h( y^* \mathbin{|} \xi^*)(\cdot)
	\end{equation*}
	for the corresponding proto-derivatives.

	``$\Rightarrow$'':
	Due to $J_B( q, u^* + t h ) = J_{\tilde B_h}(t, q)$,
	we get that $J_{\tilde B_h}$
	is directionally differentiable.
	Further, $J_{\tilde B_h}(t, q) \in Y$
	is Lipschitz continuous w.r.t.\ $q \in Y$
	and, thus,
	we get the semi-differentiability
	(see \cite[Definition~1]{AdlyRockafellar2020})
	of
	$J_{\tilde B_h}$
	at
	$q^*$.
	Now, we can apply
	\cite[Remark~5]{AdlyRockafellar2020}
	to obtain that
	$J_{\tilde B_h}$ is proto-differentiable
	at $q^*$
	relative to $y^* = J_{\tilde B_h}(0,q^*)$.
	Now,
	\cite[Lemma~2]{AdlyRockafellar2020}
	yields that
	$\tilde B_h$
	is proto-differentiable
	at
	$y^*$
	relative to
	$\xi^*$.
	As explained above,
	this yields the desired proto-differentiability of $B$.
	Moreover,
	from these arguments, we can distill the formula
	\begin{equation*}
		J_B'(q^*, u^*; \cdot, h)
		=
		J_{\tilde B_h}'(q^*; \cdot)
		=
		DJ_{\tilde B_h}(q^* \mathbin{|} y^*)
		=
		(R + D \tilde B_h (y^* \mathbin{|} \xi^* ))^{-1}
		=
		(R + D B (y^*, u^* \mathbin{|} \xi^* )(\cdot, h))^{-1}
		.
	\end{equation*}
	Since a maximal monotone mapping
	is uniquely determined
	by its resolvent,
	this shows that
	$D B (y^*, u^* \mathbin{|} \xi^* )$
	coincides with the mapping
	defined in \cref{lem:max_mon_dif_dep}
	and, therefore, is maximally monotone.

	``$\Leftarrow$'':
	Using similar arguments as above,
	this follows from
	\cite[Theorem~2]{AdlyRockafellar2020}.
\end{proof}
In particular,
the normal cone mapping to a polyhedric set
is proto-differentiable
and the proto-derivative is given
as in \cref{prop:polyhedric}.
As already mentioned, this result is known from
\cite{Do1992,Levy1999}.

\section{Quasi-generalized equations}
\label{sec:qge}
In this section,
we treat the generalization
\begin{equation}
	\tag{\ref{eq:qp}}
	0 \in A(y,u) + B(y - \Phi(y,u), u)
\end{equation}
of \eqref{eq:p}.
Here, $\Phi \colon Y \times U \to Y$ is an additional mapping.
We will follow two approaches
to investigate the solution mapping of \eqref{eq:qp}.
In the first approach,
we reformulate \eqref{eq:qp} in the form \eqref{eq:p}
by introducing a new variable $z = y - \Phi(y,u)$.
This idea was successfully
used in \cite{Wachsmuth2019:2}
to prove the directional differentiability
of QVIs.
The second approach,
which was pioneered in \cite{AlphonseHintermuellerRautenberg2019},
uses a iteration approach,
i.e., it builds a sequence $\seq{y_{u,n}}_{n \in \N}$,
in which $y_{u,n}$ solves
\begin{equation*}
	0 \in A(y_{u,n},u) + B(y_{u,n} - \Phi(y_{u,n-1},u), u)
	.
\end{equation*}

We shall see that both approaches will use a similar set of assumptions.

\begin{assumption}
	\label{asm:basic_qge}
	We assume that $(y^*, u^*) \in Y \times U$ is a solution of \eqref{eq:qp}
	such that the operators $A$ and $B$ satisfy \cref{asm:standing}.
	In addition,
	$\Phi \colon Y \times U \to Y$ is continuous at $(y^*, u^*)$
	and
	$\Phi(\cdot, u)$ is locally a uniform contraction,
	i.e., there exists a Lipschitz constant $L_\Phi \in [0,1)$
	such that
	\begin{equation*}
		\norm{\Phi(y_2, u) - \Phi(y_1, u)} \le L_\Phi \norm{y_2 - y_1}
	\end{equation*}
	for all $y_1, y_2 \in B_{\varepsilon}(y^*)$ and for all $u \in \UU$.
\end{assumption}
\begin{assumption}
	\label{asm:qge_small}
	In addition to \cref{asm:basic_qge}, we require
	\begin{enumerate}[label=(\alph*)]
		\item
			\label{asm:qge_small:1}
			$L_\Phi < \gamma_A^{-1}$,
			or
		\item
			\label{asm:qge_small:2}
			$L_\Phi < 2 \sqrt{\gamma_A} / (1 + \gamma_A)$
			and
			there exists a function $g \colon U_\varepsilon(y^*) \times \UU \to \R$
			such that
			$A$ is the Fréchet derivative of $g$
			w.r.t.\ the first variable,
	\end{enumerate}
	where $\gamma_A = L_A / \mu_A \ge 1$
	is the local condition number of $A$.
\end{assumption}
\begin{assumption}
	\label{asm:Phi_diff}
	In addition to \cref{asm:basic_qge}, we assume
	that
	\begin{enumerate}[label=(\roman*)]
		\item
			\label{asm:Phi_diff:1}
			$A$ is directionally differentiable at $(y^*, u^*)$.
		\item
			\label{asm:Phi_diff:2}
			$J_B$ is directionally differentiable at $(q^* - \phi^*, u^*)$
			with
			$q^* = y^* - R^{-1} A(y^*, u^*)$
			and
			$\phi^* = \Phi(y^*, u^*)$.
		\item
			\label{asm:Phi_diff:3}
			$\Phi$ is directionally differentiable at $(y^*, u^*)$.
	\end{enumerate}
\end{assumption}

As a preparation,
we state a consequence of \cref{asm:qge_small}.
\begin{lemma}
	\label{lem:qge_small}
	Let \cref{asm:qge_small} be satisfied.
	Then, there exist constants $C > 0$
	and $\tilde c \in (0,1)$,
	such that for all
	$u \in \UU$ and
	$y_1, y_2, z_1, z_2 \in B_\varepsilon(y^*)$
	we have
	\begin{equation*}
		\dual{
			A(z_2, u) - A(z_1, u)
		}{
			z_2 - \Phi(y_2, u) - z_1 + \Phi(y_1, u)
		}
		\ge
		C
		\parens*{
			\norm{z_2 - z_1}^2
			-
			\tilde c^2 \norm{y_2 - y_1}^2
		}
		.
	\end{equation*}
\end{lemma}
\begin{proof}
	We denote the left-hand side of the inequality by $M \in \R$.

	In case that \itemref{asm:qge_small:1}
	holds,
	we have
	\begin{align*}
		M
		\ge
		\mu_A \norm{z_2 - z_1}^2
		-
		L_A L_\Phi \norm{z_2 - z_1}\norm{y_2 - y_1}
		\ge
		\frac{\mu_A}{2} \norm{z_2 - z_1}^2
		-
		\frac{L_A^2 L_\Phi^2}{ 2 \mu_A} \norm{y_2 - y_1}^2
		,
	\end{align*}
	i.e., we can use $C = \mu_A/2$
	and $\tilde c = L_A L_\Phi / \mu_A = \gamma_A L_\Phi < 1$.

	Under
	\itemref{asm:qge_small:2}, we
	first consider $z_1, z_2 \in U_\varepsilon(y^*)$.
	\Cref{thm:cvx_lipschitz_2} yields
	\begin{align*}
		M
		&\ge
		\frac{\mu_A L_A}{\mu_A + L_A} \norm{z_2 - z_1}^2
		+
		\frac{1}{\mu_A + L_A} \norm{A(z_2, u) - A(z_1, u)}^2
		-
		L_\Phi \norm{A(z_2, u) - A(z_1, u)} \norm{y_2 - y_1}
		\\
		&
		\ge
		\frac{\mu_A L_A}{\mu_A + L_A} \norm{z_2 - z_1}^2
		-
		\frac{(\mu_A + L_A) L_\Phi^2}{4} \norm{y_2 - y_1}^2
		.
	\end{align*}
	Here, we have used Young's inequality.
	Since everything is continuous w.r.t.\ $z_1, z_2 \in Y$,
	the inequality carries over to $z_1, z_2 \in B_\varepsilon(y^*)$.
	Now,
	we can choose
	$C = \mu_A L_A / (\mu_A + L_A)$
	and
	\(
		\tilde c
		=
		(\mu_A + L_A) L_\Phi / (2 \sqrt{\mu_A L_A} )
		=
		(1 + \gamma_A) L_\Phi / ( 2 \sqrt{\gamma_A})
		<
		1
	\).
\end{proof}
Interestingly,
this already shows that (locally) \eqref{eq:qp} has at most one solution.
\begin{lemma}
	\label{lem:uniqueness}
	Let \cref{asm:qge_small} be satisfied.
	Then, for each $u \in \UU$,
	\eqref{eq:qp} has at most one solution in $B_\varepsilon(y^*)$.
\end{lemma}
\begin{proof}
	Let $y_1, y_2 \in B_\varepsilon(y^*)$ be solutions of \eqref{eq:qp}.
	Due to $-A(y_i,u ) \in B(y_i - \Phi(y_i), u)$
	we get
	\begin{equation*}
		\dual{
			A(y_2, u) - A(y_1, u)
		}{
			y_1 - \Phi(y_1, u) - y_2 + \Phi(y_2, u)
		}
		\ge 0
	\end{equation*}
	and \cref{lem:qge_small} shows
	\(
		C
		\parens{
			1
			-
			\tilde c^2
		}
		\norm{y_2 - y_1}^2
		\le
		0
	\),
	i.e.,
	$y_1 = y_2$.
\end{proof}

\subsection{Via a reformulation as GE}
\label{subsec:reformulation_GE}
In this first approach,
we reformulate \eqref{eq:qp}
via the new variable
\begin{equation}
	\label{eq:tranfo}
	z = y - \Phi(y, u).
\end{equation}
Let us first check that this transformation is locally well defined.
\begin{lemma}
	\label{lem:tranfo}
	Let \cref{asm:basic_qge} be satisfied.
	There exist neighborhoods $\ZZ \subset Y$ and $\hat\UU \subset \UU$
	of $z^* := y^* - \Phi(y^*, u^*)$ and of $u^*$, respectively,
	such that
	the equation \eqref{eq:tranfo}
	has a unique solution $y \in B_\varepsilon(y^*)$
	for all $(z, u) \in \ZZ \times \hat\UU$.

	Moreover, the mapping
	\begin{equation*}
		(\id - \Phi(\cdot, u))^{-1} \colon \ZZ \to B_\varepsilon(y^*)
	\end{equation*}
	is Lipschitz with constant $(1 - L_\Phi)^{-1}$ for all $u \in \hat\UU$.
\end{lemma}
\begin{proof}
	We define the mappings
	$\AA \colon Y \times (Y \times U) \to Y\dualspace$
	and
	$\BB \colon Y \times (Y \times U) \mto Y\dualspace$
	via
	\begin{equation*}
		\AA(y, (z,u)) := R\parens[\big]{ y - z - \Phi(y, u) },
		\qquad
		\BB(y, (z,u)) := \set{0}.
	\end{equation*}
	Now, our transformation \eqref{eq:tranfo}
	is equivalent to the generalized equation
	\begin{equation*}
		0 \in \AA(y, (z,u)) - \BB(y, (z,u))
	\end{equation*}
	and the first part of the assertion follows from \cref{thm:solvability_vi}.

	To estimate the Lipschitz constant, we take $z_1, z_2 \in \ZZ$
	and set $y_i = (\id - \Phi(\cdot, u))^{-1}(z_i)$ for $i = 1,2$.
	Then,
	\begin{equation*}
		\norm{y_2 - y_1}
		=
		\norm{z_2 + \Phi(y_2, u) - z_1 - \Phi(y_1, u)}
		\le
		\norm{z_2 - z_1} + L_\Phi \norm{y_2 - y_1}
	\end{equation*}
	and this shows the claim.
\end{proof}

Using the result of \cref{lem:tranfo},
we transform \eqref{eq:qp} into
\begin{equation}
	\label{eq:qp_z}
	0
	\in
	\tilde A(z, u)
	+
	B(z, u),
\end{equation}
where
$\tilde A \colon \ZZ \times \hat \UU \to Y$
is defined via
\begin{equation*}
	\tilde A(z, u)
	:=
	A(
		( \id - \Phi(\cdot, u))^{-1}(z)
		,
		u
	)
	.
\end{equation*}
By inserting the definitions,
we verify that \eqref{eq:qp} and \eqref{eq:qp_z}
are locally equivalent.

\begin{lemma}
	\label{lem:equiv_traf}
	Let \cref{asm:basic_qge} be satisfied.
	\begin{enumerate}[label=(\alph*)]
			\item
				If $(y,u) \in B_\varepsilon(y^*) \times \hat \UU$
				is a solution of \eqref{eq:qp}
				with
				$z := y - \Phi(y,u) \in \ZZ$,
				then
				$z$ is a solution of \eqref{eq:qp_z}.
			\item
				If $(z, u) \in \ZZ \times \hat \UU$
				is a solution of \eqref{eq:qp_z},
				then $y := (\id - \Phi(\cdot, u))^{-1}(z)$
				is a solution of \eqref{eq:qp}.
		\end{enumerate}
\end{lemma}
Note that $z \in \ZZ$ in (a)
is satisfied if $(y,u)$ is sufficiently close to $(y^*, u^*)$.

By using the ideas of \cite[Lemmas~3.3, 3.5]{Wachsmuth2019:2}, we check that the analysis from \cref{sec:gen_eq}
applies to \eqref{eq:qp_z}.
In what follows,
we use
$z^* := y^* - \Phi(y^*, u^*)$.

\begin{lemma}
	\label{lem:asm_tilde_A}
	Let \cref{asm:qge_small} be satisfied.
	Then, the operator $\tilde A \colon \ZZ \times \hat \UU \to Y$
	satisfies \cref{asm:standing}
	at
	$(z^*, u^*)$.
\end{lemma}
\begin{proof}
	Let $\eta > 0$ be given such that $B_\eta(z^*) \subset \ZZ$.
	We have to show
	the existence of $\mu_{\tilde A}, L_{\tilde A} > 0$
	such that
	\begin{align*}
		\dual{ \tilde A(z_2, u) - \tilde A(z_1, u) }{ z_2 - z_1} &\ge \mu_{\tilde A} \norm{z_2 - z_1}^2
		\\
		\norm{ \tilde A(z_2, u) - \tilde A(z_1, u)}
		&
		\le L_{\tilde A} \norm{z_2 - z_1}
	\end{align*}
	holds for all $z_1, z_2 \in B_{\eta}(z^*)$
	and for all $u \in \hat\UU$.
	The Lipschitz property follows directly from \cref{lem:tranfo}
	and it remains to prove
	the strong monotonicity.

	Let $z_1, z_2 \in B_\eta(z^*)$ and $u \in \hat\UU$
	be arbitrary
	and set
	$y_i := (\id - \Phi(\cdot, u))^{-1}(z_i) \in B_\varepsilon(y^*)$ for $i = 1,2$.
	Then,
	\cref{lem:qge_small} implies
	\begin{align*}
		\dual{ \tilde A(z_2, u) - \tilde A(z_1, u) }{ z_2 - z_1}
		&=
		\dual{ A(y_2, u) - A(y_1, u) }{ y_2 - y_1 - \Phi(y_2, u) + \Phi(y_1, u) }
		\\&\ge
		C (1 - \tilde c) \norm{y_2 - y_1}^2
		.
	\end{align*}
	In combination with
	\begin{equation*}
		\norm{z_2 - z_1}
		=
		\norm{y_2 - y_1 - \Phi(y_2, u) + \Phi(y_1, u)}
		\le
		(1 + L_\Phi) \norm{y_2 - y_1},
	\end{equation*}
	this
	shows the uniform strong monotonicity of $\tilde A$.
\end{proof}

In order to apply \cref{thm:diff_S},
we have to check that $\tilde A$
is directionally differentiable.
To this end,
we verify
that
$(z,u) \mapsto (\id - \Phi(\cdot, u))^{-1}(z)$
is directionally differentiable.

\begin{lemma}
	\label{lem:diff_psi}
	Let \cref{asm:Phi_diff} be satisfied.
	Then,
	$\Psi \colon \ZZ \times \hat\UU \to B_\varepsilon(y^*)$,
	$\Psi(z,u) := (\id - \Phi(\cdot, u))^{-1}(z)$,
	is directionally differentiable
	at $(z^*, u^*)$
	and
	the directional derivative
	$\delta = \Psi'(z^*, u^*; k, h)$
	in a direction $(k,h) \in Y \times U$
	is the unique solution of
	\begin{equation*}
		\delta
		-
		\Phi'( y^*, u^*; \delta, h)
		=
		k
		.
	\end{equation*}
\end{lemma}
\begin{proof}
	We define
	$\AA$ and $\BB$ as in the proof of \cref{lem:tranfo}.
	Then,
	$\Psi$ is the solution mapping
	$(y,u) \mapsto z$
	of
	\begin{equation*}
		0 \in \AA(y, (z,u)) - \BB(y, (z,u))
	\end{equation*}
	and
	the assertion follows from \cref{thm:diff_S}.
\end{proof}

\begin{corollary}
	\label{cor:tilde_a_direc}
	Let \cref{asm:Phi_diff} be satisfied.
	Then, the operator
	$\tilde A$ is directionally differentiable
	at $(z^*, u^*)$
	and we have
	\begin{equation*}
		\tilde A'(z^*, u^*; k, h)
		=
		A'(y^*, u^*; \Psi'(z^*, u^*; k, h), h)
		.
	\end{equation*}
\end{corollary}
If $A$ would be Lipschitz continuous w.r.t.\ both arguments,
it would be Hadamard directionally differentiable
and we could employ a general chain rule,
see, e.g., \cite[Propositions~2.47, 2.49]{BonnansShapiro2000}.
Since this is not the case, we have to adapt the proof to the situation at hand.
\begin{proof}
	We have
	\begin{equation*}
		\tilde A(z,u) = A(\Psi(z,u), u)
	\end{equation*}
	with the operator $\Psi$ from \cref{lem:diff_psi}.
	Let $(k,h) \in Y \times U$ be arbitrary.
	For $t > 0$ we have
	\begin{align*}
		&\frac{\tilde A(z^* + t k, u^* + t h) - \tilde A(z^*, u^*)}{t}
		=
		\frac{A(\Psi(z^* + t k, u^* + t h), u^* + t h) - A(\Psi(z^*, u^*), u^*)}{t}
		\\
		&\quad=
		\frac{A(\Psi(z^* + t k, u^* + t h), u^* + t h) - A(\Psi(z^*, u^*) + t \Psi'(z^*, u^*; k, h), u^* + t h)}{t}
		\\&\qquad
		+
		\frac{A(\Psi(z^*, u^*) + t \Psi'(z^*, u^*; k, h), u^* + t h) - A(\Psi(z^*, u^*), u^*)}{t}
		=:
		I_1 + I_2
		.
	\end{align*}
	Due to the Lipschitz continuity of $A$,
	the first term on the right-hand side is bounded by
	\begin{equation*}
		\norm{I_1}
		\le
		\frac{L_A}{t} \norm{\Psi(z^* + t k, u^* + t h) - \Psi(z^*, u^*) - t \Psi'(z^*, u^*; k, h)}
		\to
		0,
	\end{equation*}
	where we used the directional differentiability of $\Psi$, see \cref{lem:diff_psi}.
	Since $A$ is directionally differentiable,
	we have
	\begin{equation*}
		\lim_{t \searrow 0} I_2
		=
		A'(y^*, u^*; \Psi'(z^*, u^*; k, h), h)
	\end{equation*}
	and this shows the claim.
\end{proof}

\begin{theorem}
	\label{thm:diff_trafo}
	Suppose that \cref{asm:qge_small,asm:Phi_diff}
	are satisfied.
	Then,
	for each $h \in U$,
	there exists $t_0 > 0$ such that
	\begin{equation*}
		0 \in
		A(y_t, u + t h)
		+
		B(y_t - \Phi(y_t, u + t h), u + t h)
	\end{equation*}
	possesses a unique solution $y_t \in B_\varepsilon(y^*)$
	for all $t \in [0,t_0]$.
	Moreover, $(y_t - y^*)/t \to \delta$ as $t \searrow 0$,
	where $\delta \in Y$ is the unique solution of
	\begin{equation}
		\label{eq:diff_QGE}
		0
		\in
		A'(y^*, u^*; \delta, h)
		+
		DB(y^* - \Phi(y^*, u^*), u^* \mathbin{|} \xi^* )(\delta - \Phi'(y^*, u^*; \delta, h), h)
		,
	\end{equation}
	where $\xi^* = -A(y^*, u^*)$.
\end{theorem}
\begin{proof}
	Due to \cref{lem:equiv_traf},
	\eqref{eq:qp} is locally equivalent to \eqref{eq:qp_z}.
	Owing to the previous results,
	we can apply \cref{thm:diff_S}
	to \eqref{eq:qp_z}.
	In particular,
	\itemref{asm:diff:1}
	follows from \cref{cor:tilde_a_direc},
	whereas
	\itemref{asm:diff:2}
	requires
	directional differentiability of $J_B$
	at
	$(z^* - R^{-1} \tilde A(z^*, u^*) , u^*) = (q^* - \phi^*, u^*)$
	and this is ensured by
	\itemref{asm:Phi_diff:2}.
	If we denote by $z_t \in \ZZ$
	the local solution of \eqref{eq:qp_z}
	w.r.t.\ $u = u^* + t h$,
	then $(z_t - z^*)/t \to k$,
	where $k$ is the unique solution to
	\begin{equation*}
		0
		\in
		\tilde A'(z^*, u^*; k, h)
		+
		DB(z^*, u^* \mathbin{|} \xi^* )(k, h)
		=
		A'(y^*, u^*; \Psi'(z^*, u^*; k, h), h)
		+
		DB(z^*, u^* \mathbin{|} \xi^* )(k, h)
		.
	\end{equation*}
	It is clear that this equation is equivalent to \eqref{eq:diff_QGE}
	via the transformation
	\begin{equation*}
		k = \delta - \Phi'(y^*, u^*; \delta, h ),
		\qquad\text{i.e.,}\qquad
		\delta = \Psi'(z^*, u^*; k, h)
		.
	\end{equation*}
	Finally,
	\begin{align*}
		\frac{y_t - y^*}{t}
		&=
		\frac{
			\Psi(z_t, u + t h)
			-
			\Psi(z^*, u^*)
		}{t}
		\\&
		=
		\frac{
			\Psi(z_t, u + t h)
			-
			\Psi(z^* + t k, u + t h)
		}{t}
		+
		\frac{
			\Psi(z^* + t k, u + t h)
			-
			\Psi(z^*, u^*)
		}{t}
		\to
		\Psi'(z^*, u^*; k, h)
		=
		\delta,
	\end{align*}
	where we used that $\Psi(\cdot, u + t h)$ is Lipschitz continuous,
	see \cref{lem:tranfo}.
	The uniqueness of $y_t$ in $B_\varepsilon(y^*)$
	follows from \cref{lem:qge_small}.
\end{proof}
We mention
that
\cref{asm:qge_small}
is mainly
used
(see \cref{lem:asm_tilde_A})
to show
that 
the operator $\tilde A \colon \ZZ \times \hat \UU \to Y$
is
locally (uniformly) strongly monotone
in a neighborhood of
$(z^*, u^*)$.

\subsection{Via an iteration approach}
\label{subsec:iteration}
Here, we use a different approach to tackle
\begin{equation}
	\tag{\ref{eq:qp}}
	0 \in A(y,u) + B(y - \Phi(y,u), u)
	.
\end{equation}
For $u \in \UU$ sufficiently close to $u^*$,
we consider the sequence $\seq{y_{u,n}}_{n \in \N} \subset Y$
defined via
\begin{subequations}
	\label{eq:iterative}
	\begin{align}
		\label{eq:iterative_1}
		y_{u,0} &:= y^*,
		\\
		\label{eq:iterative_2}
		0 &\in A(y_{u,n}, u) + B(y_{u,n} - \Phi(y_{u,n-1}, u), u)
		.
	\end{align}
\end{subequations}
We will see that
this iteration
is well defined
in the sense that
\eqref{eq:iterative_2}
has a unique solution $y_{u,n} \in B_\varepsilon(y^*)$
under appropriate assumptions.
This idea was used in 
\cite{AlphonseHintermuellerRautenberg2019}
to show the directional differentiability of QVIs.
We demonstrate that this idea can also be applied to \eqref{eq:qp}.

In order to study \eqref{eq:iterative_2}
with the methods from \cref{sec:gen_eq},
we introduce the operators
$\AA \colon Y \times (Y \times U) \to Y\dualspace$,
$\BB \colon Y \times (Y \times U) \mto Y\dualspace$
via
\begin{equation}
	\label{eq:op_AA_BB}
	\AA(y, (\phi,u)) := A(y, u),
	\qquad
	\BB(y, (\phi,u)) := B(y - \phi, u)
	.
\end{equation}
Moreover, we set
$\phi^* := \Phi(y^*, u^*)$.
Then, under \cref{asm:basic_qge},
it is clear that \cref{asm:standing}
is satisfied by $(\AA, \BB)$ at $(y^*, (\phi^*, u^*))$.

Moreover,
for arbitrary $(q, (\phi, u)) \in Y \times (Y \times \UU)$ and $\rho > 0$,
the point $y = J_{\rho \BB}(q, (\phi, u))$
solves
\begin{equation*}
	0
	\in
	R (y - q) + \rho \BB(y, (\phi, u))
	=
	R( (y-\phi) - (q-\phi) ) + \rho B(y - \phi, u),
\end{equation*}
i.e.,
we have the relation
\begin{equation}
	\label{eq:res_BB}
	J_{\rho \BB}(q, (\phi, u))
	=
	J_{\rho B}(q - \phi, u) + \phi
\end{equation}
between the resolvents of $B$ and $\BB$.

Next, we address the local solvability of \eqref{eq:iterative_2}.
\begin{lemma}
	\label{lem:solve_fix}
	Let \cref{asm:basic_qge} be satisfied and fix $\rho$, $c$ and $q_\rho^*$ as in \cref{thm:solvability_vi}.
	Then,
	for all
	$r  \in (0,\varepsilon]$,
	$(\phi, u) \in Y \times \UU$
	with
	\begin{equation*}
			2 \norm{\phi - \phi^*}
			+
			\norm{
				J_{\rho B}(q_\rho^* - \phi^*, u)
				-
				J_{\rho B}(q_\rho^* - \phi^*, u^*)
			}
			+
			\rho \norm{A(y^*, u) - A(y^*, u^*)}
		\le
		(1-c) r
	\end{equation*}
	the equation
	\begin{equation*}
		0 \in A(z, u) + B(z - \phi, u)
	\end{equation*}
	has a unique solution $z \in B_r(y^*)$.
\end{lemma}
\begin{proof}
	The equation can be recast as
	\begin{equation*}
		0 \in \AA(y, (\phi, u)) + \BB(y, (\phi, u))
	\end{equation*}
	and we are going to apply \cref{thm:solvability_vi} with $\zeta = 0$.
	It is clear that \cref{asm:standing}
	is satisfied by $(\AA, \BB)$
	and the operator $\AA$ possesses the same constants as $A$.
	Thus, it remains to show that
	\begin{equation*}
		\norm{
			J_{\rho \BB}(q_\rho^*, (\phi,u))
			-
			J_{\rho \BB}(q_\rho^*, (\phi^*, u^*))
		}
		+
		\rho \norm{\AA(y^*, (\phi,u)) - \AA(y^*, (\phi^*,u^*))}
		\le
		(1 - c) r
	\end{equation*}
	is satisfied.
	This, however, follows from the estimate
	\begin{align*}
		&\norm{
			J_{\rho \BB}(q_\rho^*, (\phi,u))
			-
			J_{\rho \BB}(q_\rho^*, (\phi^*, u^*))
		}
		=
		\norm{
			J_{\rho B}( q_\rho^* - \phi, u) + \phi
			-
			J_{\rho B}( q_\rho^* - \phi^*, u^*) - \phi^*
		}
		\\
		&\qquad
		\le
		\norm{
			J_{\rho B}( q_\rho^* - \phi, u)
			-
			J_{\rho B}( q_\rho^* - \phi^*, u)
		}
		+
		\norm{
			J_{\rho B}( q_\rho^* - \phi^*, u)
			-
			J_{\rho B}( q_\rho^* - \phi^*, u^*)
		}
		+
		\norm{\phi - \phi^*}
		\\
		&\qquad
		\le
		2 \norm{\phi - \phi^*}
		+
		\norm{
			J_{\rho B}( q_\rho^* - \phi^*, u)
			-
			J_{\rho B}( q_\rho^* - \phi^*, u^*)
		}
		.
		\qedhere
	\end{align*}
\end{proof}
\begin{lemma}
	\label{lem:solve_fix_2}
	Let \cref{asm:basic_qge} be satisfied and fix $\rho$, $c$ and $q_\rho^*$ as in \cref{thm:solvability_vi}.
	Then, there exists a constant $\lambda \in (0,\varepsilon]$,
	such that
	for all
	$y \in B_\lambda(y^*)$
	and
	all $u \in \UU$
	with
	\begin{equation}
		\label{eq:est_phi_u}
		\begin{aligned}
			&
			\CC_\rho(u)
			:=
			2 \norm{ \Phi(y^*,u) - \Phi(y^*, u^*) }
			+
				\norm{
					J_{\rho B}(q_\rho^* - \phi^*, u)
					-
					J_{\rho B}(q_\rho^* - \phi^*, u^*)
				}
				\\&\mspace{350mu}
				+
				\rho \norm{A(y^*, u) - A(y^*, u^*)}
			\le
			\frac{1-c}{2} \varepsilon
		\end{aligned}
	\end{equation}
	the equation
	\begin{equation*}
		0 \in A(z, u) + B(z - \Phi(y, u), u)
	\end{equation*}
	has a unique solution $z := T_u(y) \in B_\varepsilon(y^*)$.
\end{lemma}
\begin{proof}
	Using
	\begin{align*}
		\norm{ \Phi(y,u) - \phi^* }
		&\le
		\norm{ \Phi(y,u) - \Phi(y^*, u) }
		+
		\norm{ \Phi(y^*,u) - \Phi(y^*, u^*) }
		\\&
		\le
		L_\Phi \norm{y - y^*}
		+
		\norm{ \Phi(y^*,u) - \Phi(y^*, u^*) }
		\le
		L_\Phi \lambda
		+
		\norm{ \Phi(y^*,u) - \Phi(y^*, u^*) }
	\end{align*}
	the assertion follows from \cref{lem:solve_fix}
	with
	$\lambda = (1 - c) \varepsilon / (4 L_\Phi)$.
\end{proof}
In the next lemma,
we apply the Banach fixed-point theorem
to $T_u$
in order to show the convergence of \eqref{eq:iterative}.
\begin{lemma}
	\label{lem:BFP_it}
	Let \cref{asm:qge_small} be satisfied and fix $\rho$, $c$, $q_\rho^*$ as in \cref{thm:solvability_vi}
	and choose $\lambda$ according to \cref{lem:solve_fix_2}.
	\begin{enumerate}[label=(\alph*)]
		\item
			\label{lem:BFP_it:1}
			There exists a constant $\tilde c \in (0,1)$,
			such that $T_u \colon B_\lambda(y^*) \to B_\varepsilon(y^*)$
			is Lipschitz continuous with modulus $\tilde c$
			for all $u \in \UU$.
		\item
			\label{lem:BFP_it:2}
			If $u \in \UU$ is chosen such that
			$\CC_\rho(u) \le (1-c) \min \set{ (1-\tilde c) \lambda, \varepsilon }$,
			then
			$T_u$ maps $B_\lambda(y^*)$ to $B_\lambda(y^*)$.
			Moreover,
			the sequence $\seq{y_{u,n}}_{n \in \N}$ given by the iteration \eqref{eq:iterative}
			satisfies
			\begin{equation*}
				\norm{ y_{u,n} - y_u }
				\le
				\frac{ \tilde c^n }{(1- c) (1-\tilde c)} \CC_\rho(u)
				,
			\end{equation*}
			where $y_u \in B_\lambda(y^*)$
			is the solution of \eqref{eq:qp}.
	\end{enumerate}
\end{lemma}
\begin{proof}
	We start by proving \ref{lem:BFP_it:1}.
	Let $y_1, y_2 \in B_\lambda(y^*)$ be given and set $z_i := T_u(y_i) \in B_\varepsilon(y^*)$, $i = 1, 2$.
	Then, $-A(z_i, u) \in B(z_i - \Phi(y_i, u), u)$
	and, thus, the monotonicity of $B$ yields
	\begin{equation*}
		\dual{
			A(z_2, u) - A(z_1, u)
		}{
			z_1 - \Phi(y_1, u) - z_2 + \Phi(y_2, u)
		}
		\ge 0
		.
	\end{equation*}
	Consequently,
	\cref{lem:qge_small}
	implies
	\begin{equation*}
		0
		\ge
		C (\norm{z_2 - z_1}^2 - \tilde c^2 \norm{y_2 - y_1}^2),
	\end{equation*}
	i.e., $\norm{z_2 - z_1}^2 \le \tilde c \norm{y_2 - y_1}$.

	Now, let $u \in \UU$ be chosen as in \ref{lem:BFP_it:2}.
	This enables us to apply
	\cref{lem:solve_fix}
	with
	the choices
	$r = (1-c)^{-1} \CC_\rho(u) \le \varepsilon$ and $\phi = \Phi(y^*, u)$.
	This shows that $T_u(y^*)$, which is the solution of
	$0 \in A(z, u) + B(z - \Phi(y^*, u), u)$,
	satisfies $\norm{ T_u(y^*) - y^* } \le r \le (1-\tilde c) \lambda$.
	Consequently, every $y \in B_\lambda(y^*)$ satisfies
	\begin{equation*}
		\norm{T_u(y) - y^*}
		\le
		\norm{T_u(y) - T_u(y^*)}
		+
		\norm{T_u(y^*) - y^*}
		\le
		\tilde c \lambda + (1-\tilde c) \lambda
		=
		\lambda.
	\end{equation*}
	Thus, we can apply the Banach fixed-point theorem to
	obtain the existence of $y_u \in B_\lambda(y^*)$.
	Due to $y_{0,n} = y^*$ and $y_{1,n} = T_u(y^*)$,
	this also yields the a-priori estimate
	\begin{equation*}
		\norm{ y_{u,n} - y_u }
		\le
		\frac{\tilde c^n}{1 - \tilde c} \norm{y_{u,1} - y_{u,0} }
		\le
		\frac{ \tilde c^n }{(1- c) (1-\tilde c)} \CC_\rho(u)
		.
		\qedhere
	\end{equation*}
\end{proof}
The next lemma helps us to control the term
$\CC_\rho(u^* + t h)$.
\begin{lemma}
	\label{lem:control_cc}
	Let \cref{asm:Phi_diff} be satisfied
	and fix $h \in U$.
	Then, for any $\rho > 0$, there exist
	constants $C \ge 0$ and $t_0 > 0$
	such that
	\begin{equation*}
		\CC_\rho( u^* + t h )
		\le
		C t
		\qquad\forall t \in [0,t_0].
	\end{equation*}
\end{lemma}
\begin{proof}
	With
	$q_\rho^* = y^* - \rho R^{-1} A(y^*, u^*)$
	and
	$\phi^* = \Phi(y^*, u^*)$
	we have
	$y^* - \phi^* = J_{\rho B}( q_\rho^* - \phi^*, u^*)$.
	Owing to \cref{cor:dir_dif_res},
	the directional differentiability
	of
	$J_{\rho B}$
	at
	$( q_\rho^* - \phi^*, u^*)$
	follows
	from the directional differentiability
	of
	$J_B$
	at
	$(q^* - \phi^*, u^*)$
	and this is guaranteed by
	\cref{asm:Phi_diff}.
	Hence, for all terms appearing in the definition \eqref{eq:est_phi_u}
	of $\CC_\rho(u^* + t h)$,
	we can utilize the directional differentiabilities
	of the involved operators
	and this yields the desired estimate.
\end{proof}

\begin{theorem}
	\label{thm:diff_via_it}
	Let \cref{asm:qge_small,asm:Phi_diff} be satisfied.
	For all $h \in U$ there exists $t_0 > 0$,
	such that
	for all $t \in [0,t_0]$,
	the equation
	\begin{equation*}
		0 \in
		A(y_t, u + t h)
		+
		B(y_t - \Phi(y_t, u + t h), u + t h)
	\end{equation*}
	has a unique solution $y_t \in B_\lambda(y^*)$,
	where $\lambda$ is chosen as in \cref{lem:solve_fix_2}.
	Moreover,
	the difference quotient $(y_t - y^*)/t$
	converges strongly in $Y$ towards $\delta \in Y$
	which is the unique solution of the linearized equation
	\begin{equation*}
		\tag{\ref{eq:diff_QGE}}
		0
		\in
		A'(y^*, u^*; \delta, h)
		+
		DB(y^* - \Phi(y^*, u^*), u^* \mathbin{|} \xi^* )(\delta - \Phi'(y^*, u^*; \delta, h), h)
		,
	\end{equation*}
	where $\xi^* = -A(y^*, u^*)$.
\end{theorem}
\begin{proof}
	We fix $h \in U$.
	Then, \cref{lem:BFP_it,lem:control_cc}
	imply the existence of $t_0 > 0$, such that
	for all $t \in [0,t_0]$
	the sequence $\seq{y_{t,n}}_{n \in \N}$ defined via $y_{t,0} := y^*$
	and each
	$y_{t,n} \in B_\varepsilon(y^*)$ solves
	\begin{equation*}
		0 \in A(y_{t,n},u^* + t h) + B(y_{t,n} - \Phi(y_{t, n-1},u^* + t h), u^* + t h)
		,
	\end{equation*}
	satisfies
	\begin{equation}
		\label{eq:uniform_limit}
		\norm{ y_{t,n} - y_t }
		\le
		C t \tilde c^n
	\end{equation}
	where $\tilde c \in (0,1)$ is as in \cref{lem:BFP_it}
	and
	$C > 0$ is a constant.

	Next, we study the differentiability of $y_{t,n}$ w.r.t.\ $t > 0$.
	We claim that for all $n \ge 0$, we have the directional differentiabilities
	\begin{equation}
		\label{eq:diff_induction}
		\frac{y_{t,n} - y_t}{t} \to \delta_n,
		\qquad
		\frac{ \Phi(y_{t, n}, u^* + t h) - \phi^*}{t}
		\to
		\Phi'(y^*, u^*; \delta_{n}, h)
		,
	\end{equation}
	where
	$\delta_0 = 0$
	and for $n \ge 1$
	the point
	$\delta_n \in Y$ solves
	\begin{equation}
		\label{eq:QVI_delta_n}
		0
		\in
		A'(y^*, u^*; \delta_n, h)
		+
		DB(y^* - \phi^*, u^* \mathbin{|} \xi^* )(\delta_n - \Phi'(y^*, u^*; \delta_{n-1}, h), h)
		.
	\end{equation}
	We argue by induction over $n$.
	The base case $n = 0$ is clear since $y_{t,0} = y^* = y_0$.
	Assume that the assertion holds for $n - 1$.
	We abbreviate
	$\phi_{t,n} := \Phi(y_{t, n}, u^* + t h)$.
	Using the operators from \eqref{eq:op_AA_BB},
	we can recast the equation for $y_{t,n}$ as
	\begin{equation*}
		0 \in
		\AA( y_{t,n}, (\phi_{t,n-1}, u^* + th ))
		+
		\BB( y_{t,n}, (\phi_{t,n-1}, u^* + th ))
		.
	\end{equation*}
	Next, we apply \cref{lem:solve_fix} for $t > 0$ small enough (depending on $n$)
	with $\phi := \phi^* + t \psi_{t,n}$, $\psi_{t,n} = \Phi'(y^*, u^*; \delta_{n-1}, h)$
	to obtain a solution
	$\tilde y_{t,n} \in B_\varepsilon(y^*)$ of
	\begin{equation*}
		0 \in A(\tilde y_{t,n},u^* + t h) + B(\tilde y_{t,n} - \phi^* - t \Phi'(y^*,u^*; \delta_{n-1}, h), u^* + t h)
	\end{equation*}
	or, equivalently,
	\begin{equation*}
		0 \in
		\AA( \tilde y_{t,n}, (\phi^* + t \psi_{t,n}, u^* + th ))
		+
		\BB( \tilde y_{t,n}, (\phi^* + t \psi_{t,n}, u^* + th ))
		.
	\end{equation*}
	Now, we are in position to apply \cref{thm:diff_S}
	and this yields
	$(\tilde y_{t,n} - y^*) / t \to \delta_n$
	in $Y$ as $t \searrow 0$,
	where
	$\delta_n \in Y$ solves
	\begin{equation*}
		0 \in
		\AA'( y^*, (\phi^*, u^*); \delta_n, (\psi_{t,n}, h))
		+
		D\BB( y^*, (\phi^*, u^*); \delta_n, (\psi_{t,n}, h))
		.
	\end{equation*}
	Using \cref{lem:max_mon_dif_dep} and \eqref{eq:res_BB},
	we can relate $DB$ and $D\BB$ as follows
	\begin{align*}
		&D\BB( y^*, (\phi^*, u^*) \mathbin{|} \xi^*)(\delta, (\psi, h))
		\\
		&\qquad=
		\set[\big]{
			R (k - \delta)
			\given
			k \in Y
			,
			\delta = J_{\BB}'(q^*, (\phi^*, u^*); k, (\psi, h))
		}
		\\
		&\qquad=
		\set[\big]{
			R (k - (\delta-\psi))
			\given
			k \in Y
			,
			\delta - \psi = J_B'(q^* - \phi^*, u^*; k - \psi, h)
		}
		\\
		&\qquad=
		DB(y^* - \phi^*, u^* \mathbin{|} \xi^* )(\delta - \psi, h)
		.
	\end{align*}
	This results in the equation \eqref{eq:QVI_delta_n}.
	By using the equations satisfied by $y_{t,n}$
	and $\tilde y_{t,n}$,
	we get
	\begin{align*}
		&
		\dual{
			A(\tilde y_{t,n},u^* + t h)
			-
			A(y_{t,n},u^* + t h)
		}{
			\tilde y_{t,n}
			-
			y_{t,n}
		}
		\\&\quad
		\le
		\dual{
			A(\tilde y_{t,n},u^* + t h)
			-
			A(y_{t,n},u^* + t h)
		}{
			\phi^* + t \Phi'(y^*,u^*; \delta_{n-1}, h)
			- \Phi(y_{t, n-1},u^* + t h)
		}
		.
	\end{align*}
	Consequently,
	\begin{equation*}
		\frac1t \norm{ \tilde y_{t,n} - y_{t,n} }
		\le
		\frac{L_A}{\mu_A t} \norm{
			\phi^* + t \Phi'(y^*,u^*; \delta_{n-1}, h)
			- \Phi(y_{t, n-1},u^* + t h)
		}
		\to
		0
		\quad\text{as } t \searrow 0
		.
	\end{equation*}
	In combination with $(\tilde y_{t,n} - y^*) / t \to \delta_n$,
	this yields the first convergence in \eqref{eq:diff_induction}.
	The second convergence in \eqref{eq:diff_induction} follows
	since $\Phi$ is directionally differentiable and Lipschitz w.r.t.\ its first argument.
	Consequently,
	\eqref{eq:diff_induction} holds for all $n \ge 0$.

	Thus, we have shown
	\begin{equation}
		\label{eq:some_limits}
		\lim_{n \to \infty} \frac{y_{t,n} - y^*}{t} = \frac{y_t - y^*}{t}
		\quad\text{and}\quad
		\lim_{t \searrow 0} \frac{y_{t,n} - y^*}{t} = \delta_n,
	\end{equation}
	where both limits exist (strongly) in $Y$.
	Moreover,
	\begin{equation*}
		\norm*{
			\frac{y_{t,n} - y^*}{t} - \frac{y_t - y^*}{t}
		}
		=
		\norm*{
			\frac{y_{t,n} - y_t}{t}
		}
		\le
		C \tilde c^n,
	\end{equation*}
	cf.\ \eqref{eq:uniform_limit}.
	This shows that the limit $n \to \infty$ in \eqref{eq:some_limits} is uniform in $t \in (0,t_0]$.
	Hence, the classical theorem on the existence and equality of iterated limits
	ensures
	\begin{equation*}
		\delta
		:=
		\lim_{t \searrow 0} \frac{y_{t} - y^*}{t}
		=
		\lim_{t \searrow 0} \lim_{n \to \infty} \frac{y_{t,n} - y^*}{t}
		=
		\lim_{n \to \infty} \lim_{t \searrow 0} \frac{y_{t,n} - y^*}{t}
		=
		\lim_{n \to \infty} \delta_n
		\quad\text{in } Y.
	\end{equation*}
	Finally, passing to the limit $n \to \infty$ in \eqref{eq:QVI_delta_n}
	yields
	the equation for $\delta$.
\end{proof}

Note that
\cref{asm:qge_small} is only used in \cref{lem:BFP_it}.
It can be replaced by requiring that
the map
$T_u \colon B_\lambda(y^*) \to B_\varepsilon(y^*)$
is a contraction uniformly in $u \in \UU$.

It is quite interesting to see that the approaches from
\cref{subsec:reformulation_GE,subsec:iteration}
use the same assumptions
and,
actually,
\cref{thm:diff_trafo,thm:diff_via_it}
coincide.
If we look a little bit more carefully,
we see that in \cref{subsec:reformulation_GE},
\cref{lem:qge_small}
is only applied in the special case $y_i = z_i$, $i = 1, 2$,
see \cref{lem:uniqueness,lem:asm_tilde_A},
whereas \cref{subsec:iteration}
requires the application in the general case
$y_i \ne z_i$, $i = 1,2$,
see \cref{lem:BFP_it}.

Altogether,
it seems to be possible to craft special situations
in which only one of the approaches of
\cref{subsec:reformulation_GE,subsec:iteration}
is applicable.
However,
we think that
(up to some exceptional boundary cases)
the ranges of applicability
of both approaches coincide.

\section{Applications}
\label{sec:applications}

\subsection{Optimization with a parameter-dependent sparsity functional}
As a first application,
we
consider
the minimization problem
\begin{equation}
	\label{eq:sparse}
	\tag{$P(u)$}
	\text{Minimize}\quad
	F(y)
	+
	G(y,u)
	\quad
	\text{w.r.t.\ } y \in Y
\end{equation}
with a parameter $u \in U$.
Here,
$(\Omega, \Sigma, \mu)$
is a measure space
and
$Y = U = L^2(\mu)$.
Moreover,
$F \colon Y \to \R$
is a given functional
and
$G \colon Y \times U \to \R$
is defined via
\begin{equation*}
	G(y,u)
	:=
	\int_\Omega \abs{u y} \, \d\mu
	.
\end{equation*}
Thus,
\eqref{eq:sparse}
models, e.g., optimal control problems
which include a sparsity functional
and we are interested
in the sensitivity of the solution $y$
w.r.t.\ the distributed sparsity parameter $u$.

Suppose that $u^* \in U$ is fixed
and $y^* \in Y$ is a local minimizer
of \hyperref[eq:sparse]{\textup{($P(u^*)$)}}.
We assume that $F$
is Fréchet differentiable
in $B_\varepsilon(y^*)$ for some $\varepsilon > 0$,
such that its derivative $F'$
is (uniformly) strongly monotone and Lipschitz continuous
on $B_\varepsilon(y^*)$,
see \itemref{asm:standing:1}.

As it is usually done,
we identify the dual space of $Y = L^2(\mu)$
with itself.

Due to the convexity of $F$,
a point $y \in Y$ with $\norm{y - y^*} < \varepsilon$
is a local minimizer of \eqref{eq:sparse}
with $u \in U$
if and only if
\begin{equation*}
	0 \in F'(y) + \partial_y G(y,u)
	,
\end{equation*}
where
\begin{equation*}
	\partial_y G(y,u)
	=
	\set*{
		g \in L^2(\Omega)
		\given
		G(v,u) \ge G(y,u) + \int_\Omega g (v - y) \, \d\mu
		\quad\forall v \in Y
	}
\end{equation*}
is the subdifferential of $G$ w.r.t.\ $y$.
It is clear that \cref{asm:standing}
is satisfied with the setting
\begin{equation*}
	A(y,u) := F'(y),
	\qquad
	B(y,u) := \partial_y G(y,u)
	.
\end{equation*}
In order to apply the results from \cref{sec:gen_eq},
we have to study the properties of the resolvent $J_{\rho B}$, $\rho > 0$.
It is clear that
\begin{equation*}
	J_{\rho B}(q,u)
	=
	\prox_{\rho G(\cdot, u)}(q)
	=
	\argmin_{v \in Y} \int_\Omega \frac12 (v - q)^2 + \rho \abs{u v} \, \d\mu
	.
\end{equation*}
Now, a pointwise discussion shows that the resolvent
can be computed pointwise and is given by a soft-shrinkage with parameter $\rho \abs{u}$,
i.e.,
\begin{equation*}
	J_{\rho B}(q,u) (x)
	=
	\shrink_{\rho \abs{u(x)}}(q(x))
	:=
	\max( \abs{q(x)} - \rho \abs{u(x)}, 0) \sign(q(x))
	.
\end{equation*}
Now, since
\begin{equation*}
	\R^2 \ni (q,u) \mapsto
	\shrink_{\rho \abs{u}}(q)
	=
	\max( \abs{q} - \rho \abs{u}, 0) \sign(q)
	\in
	\R
\end{equation*}
is Lipschitz continuous and directionally differentiable,
it is easy to check that also
the associated Nemytskii operator
$J_{\rho B} \colon Y \times U \to Y$
is Lipschitz continuous and directionally differentiable.
If, additionally,
$F' \colon Y \to Y$
is directionally differentiable,
we are in position to apply \cref{thm:diff_S}
to obtain the directional differentiability
of the (local) solution mapping of \eqref{eq:sparse}.
Using \cref{lem:max_mon_dif_dep},
it is also possible to characterize the directional derivative.

\subsection{Quasi-linear QVIs}
We demonstrate
the applicability of our results to
a QVI governed by a quasi-linear operator.
To this end,
let
$\Omega \subset \R^d$ be open and bounded
and
with
$Y = H_0^1(\Omega)$
and
$U = L^2(\Omega)$
we define the
quasi-linear operator
$A \colon Y \times U \to Y\dualspace$
via
\begin{equation*}
	A(y,u)
	:=
	-\div g(\nabla y,u)
	+
	f(u)
	,
	\quad\text{i.e.}\quad
	\dual{A(y,u)}{v}
	=
	\int_\Omega g(\nabla y, u) \nabla v + f(u) v \, \d x,
\end{equation*}
where
$g \colon \R^d \times \R \to \R^d$
is (uniformly) strongly monotone
w.r.t.\ its first argument;
and Lipschitz continuous and differentiable on $\R^d \times \R$.
Moreover,
$f \colon \R \to \R$ is differentiable and Lipschitz continuous.
These conditions imply that $A$ satisfies \cref{asm:standing}
globally on $Y$.
Moreover,
using the dominated convergence theorem,
we can check that $A$ is directionally differentiable
with
\begin{equation*}
	\dual{
		A'(y, u; \delta, h)
	}{v}
	=
	\int_\Omega \parens*{ g'_y(\nabla y, u) \nabla \delta + g'_u(\nabla y, u) h } \nabla v + f'(u) h v \, \d x
	,
\end{equation*}
see also \cite[Theorem~8]{GoldbergKampowskyTroeltzsch1992}.

To define the operator $B$,
let
$K \subset Y$
be given
such that $\Proj_K$ is directionally differentiable,
e.g., we could choose a polyhedric $K$.
We set
$B(\cdot) := N_K(\cdot)$,
where $N_K$ is the normal cone mapping of $K$.
Note that $B$ is independent of the variable $u$.

With this setting, \cref{asm:standing,asm:diff} are satisfied.
Next, we choose $\Phi \colon Y \times U \to Y$
such that
there exists a Lipschitz constant $L_\Phi \in [0,1)$
with
\begin{equation*}
	\norm{\Phi(y_2, u) - \Phi(y_1, u)} \le L_\Phi \norm{y_2 - y_1}
\end{equation*}
for all $y_1, y_2 \in Y$ and all $u \in U$.
Further, we suppose that $\Phi$ is directionally differentiable
and that \cref{asm:qge_small} concerning the smallness of $L_\Phi$ is satisfied.

Since our assumptions on $A$, $B$ and $\Phi$ are global,
we obtain that for all $u \in U$,
there exists a unique solution $y \in Y$
of
the QVI
\begin{equation*}
	0 \in A(y,u) + B(y - \Phi(y, u)),
\end{equation*}
cf.\ \cite[Section~3]{Wachsmuth2019:2}.
Since all the assumptions from \cref{sec:qge}
are satisfied,
our differentiability theorems
imply
that the mapping $y \mapsto u$
is directionally differentiable.
For a fixed parameter $u^* \in U$
we denote the solution by $y^* \in Y$.
Then,
the directional derivative $\delta \in Y$
in the direction $h \in U$
is given by the solution of
\begin{equation*}
	0
	\in
	A'(y^*, u^*; \delta, h)
	+
	DB(y^* - \Phi(y^*, u^*) \mathbin{|} \xi^* )(\delta - \Phi'(y^*, u^*; \delta, h))
	,
\end{equation*}
where
$\xi^* = -A(y^*, u^*)$.
We mention that in the particular case that the set $K$ is polyhedric,
the set-valued mapping
$DB(y^* - \Phi(y^*, u^*) \mathbin{|} \xi^* )$
coincides with the normal cone mapping of the critical cone
\(
	\KK
	=
	T_K(y^*) \cap (\xi^*)\anni
\),
see \cref{prop:polyhedric}.

It is also clear that the above assumptions and arguments
can be localized if we already have a solution $y^*$
of the QVI
corresponding to the parameter $u^*$.

\bibliographystyle{jnsao}
\bibliography{from_resolvents_to_qvis}

\end{document}